\documentclass{elsarticle}
\title{On a set-valued Young integral with applications to differential inclusions}

\usepackage{amsfonts,amssymb,amsmath,amsthm,mathtools}
\usepackage{url} % pour bibtex
\usepackage{enumerate}
\usepackage{xcolor}% pour les corrections 
\usepackage[normalem]{ulem} % pour les corrections (l'option normalem
\allowdisplaybreaks

\newtheorem{theorem}{Theorem}[section]

\newtheorem{proposition}[theorem]{Proposition}
\newtheorem{corollary}[theorem]{Corollary}

\theoremstyle{definition}
\newtheorem{definition}[theorem]{Definition}

\newtheorem{remark}[theorem]{Remark}
\newtheorem{example}[theorem]{Example}
\newtheorem{notation}[theorem]{Notation}
\newcommand\R{\mathbb{R}}
\newcommand\N{\mathbb{N}}
\newcommand\C{\mathrm{C}} % espaces de fonctions (Hölder etc)
\newcommand\D{\mathrm{D}} % Demyanov
\newcommand\Dem{\mathrm{Dem}} % Demyanov
\newcommand\Haus{\mathrm{H}} % Hausdorff
\newcommand\XXX{\mathrm{X}} % Hausdorff ou Demyanov
\newcommand\norm[1]{\left\lVert #1 \right\rVert}
\newcommand\abs[1]{\left\lvert #1 \right\rvert}
\newcommand\CCO[1]{\left( #1 \right)}
\newcommand\tq{\textrm{ $;$ }}

\newcommand\cc{\mathcal{P}_{\mathrm{ck}}}  % compacts convexes
\newcommand\cl{\mathcal{P}_{\mathrm{f}}} % {\mathcal{F}} % fermés
\newcommand\komp{\mathcal{P}_{\mathrm{k}}}  % compacts

\newcommand\proj[1]{\pi_{#1}} % projection orthogonale sur #1

\newcommand\Aint[1]{(\mathrm{A}_{#1})\!\int} % intégrale d'Aumann-Young de
\newcommand\Cte{\mathfrak{C}}
\newcommand\RHspace{\mathbb{S}} % plongement de Radstrom-Hormander

\newcommand\matr{M_{e,d}(\mathbb R)}
\newcommand\matrl{M_{\ell,d}(\mathbb R)}
\newcommand{\ccm}{\cc(\matr)}

\newcommand\St{\mathrm{St}} %{ { \normalfont{\textrm{St}} } }
\newcommand\StLip[1]{\mathfrak{k}_{#1}} % constante de Lipschitz pour St

\newcommand\M{\mathbb{M}} % un espace métrique
\newcommand{\Domain}{\mathop{\mathrm{Dom}}}
\newcommand{\Li}{\mathop{\mathrm{Li}}\limits}
\newcommand\Ls{\mathop{\mathrm{Ls}}\limits}

\newcommand\HA{\textrm{(A)}} % hypothèses
\newcommand\HB{\textrm{(B)}}
\newcommand\rmin{r_{\min}}
\newcommand\B{\mathbb{B}}

\begin{document}

%%%% FRONT PAGE %%%%%%%%%%%%%%%%%%%%%%%%%%%%%
\begin{frontmatter}
  
\author[1]{Laure \textsc{Coutin}}
\ead{laure.coutin@math.univ-toulouse.fr}
%%%%%%%
\author[2]{Nicolas \textsc{Marie}}
\ead{nmarie@parisnanterre.fr}
%%%%%%%%
\author[3]{Paul \textsc{raynaud de Fitte}} %\corref{cor1}
\ead{prf@univ-rouen.fr}

\address[1]{IMT, Universit\'e Paul Sabatier, Toulouse, France}
\address[2]{MODAL'X, Universit\'e Paris Nanterre, Nanterre, France}
\address[3]{LMRS, Universit\'e de Rouen Normandie, Rouen, France}

\begin{abstract}
We present a new Aumann-like integral for a Hölder multifunction with respect
to a Hölder signal, based on the Young integral of a particular set of
Hölder selections. This restricted Aumann integral has
continuity properties that allow for numerical approximation 
as well as an existence theorem
for an abstract stochastic differential inclusion.
This is applied to concrete examples of
first order and second order stochastic differential inclusions 
directed by fractional Brownian motion.   
\end{abstract}

\begin{keyword}
  set-valued integral \sep Aumann integral \sep Young integral 
\end{keyword}

\end{frontmatter}
% \maketitle

%%%%%%%%%%%%%%%%%%%%%%%%%%%%%%%%%%%%%%%%%%%
%\tableofcontents
%

% Section : Introduction.

%
\section{Introduction}
Consider $d,e\in\mathbb N^*$, a $\beta$-H\"older continuous signal $w
: [0,T]\rightarrow\mathbb R^d$ with $\beta\in (0,1)$, and an
$\alpha$-H\"older continuous multifunction $F$ defined on $[0,T]$ with
convex compact values in the set $\matr$ of linear mappings from
$\R^d$ to $\R^e$,
where  
$\alpha\in (0,1)$ and $\alpha +\beta > 1$.
The purpose of this paper is to define a set-valued Young integral
of $F$ with respect to $w$, 
of
the form
\begin{equation}\label{integral_introduction}
\int_{0}^{T}F(s)\,dw(s) =
\left\{\int_{0}^{T}f(s)\,dw(s)
\tq
f\in\mathcal S(F)\right\}
\end{equation}
in a nontrivial way, but with a small enough set of selections $\mathcal S(F)$,
so as 
to get algebraic and topological properties (convexity, boundedness,
compactness, continuity, etc.) allowing to give a sense and establish
the existence of solutions to several types of differential
inclusions.

%% UN PEU D'HISTOIRE...
There have been many different approaches to set-valued
integration with respect to a nonnegative $\sigma$-additive measure $\mu$.
The most popular one is due to Aumann \cite{Aumann}, based
on Lebesgue integrals of selections:
\begin{equation}
\int_{0}^{T}F(s)\,d\mu(s) =
\left\{\int_{0}^{T}f(s)\,d\mu(s)
\tq
f\in\mathcal{S}_{L^1(\mu)}(F) \right\}
\end{equation}
where $\mathcal{S}_{L^1(\mu)}(F)$ denotes the set of all
$\mu$-integrable selections of $F$. 
In the case of multifunctions with convex
compact values, other approaches such as Hukuhara's \cite{Hukuhara} or
Debreu's \cite{Debreu} are restricted to multifunctions with compact
convex values and use the cone structure of the space of convex
compact sets.

The concept of Aumann integral has been applied in several papers
to integration of multivalued stochastic processes
using classical %(It\^o)
stochastic calculus
and a definition of the form \eqref{integral_introduction},
where $w$ is a Brownian motion, or
more generally a semimartingale, e.g., \cite{kisielewicz1997,MM11,
  malinowski-michta-sobolewska}. 
Despite this similarity,
the setting of stochastic calculus
is quite different from ours, 
since the stochastic integral needs a probability space to make
sense.
In this context, some variants have been developped, the main one by
Jung and Kim \cite{jung-kim}, which is the
decomposable hull of an integral of the form
\eqref{integral_introduction}, 
has been extensively studied by Polish mathematicians from Zielona
G\'ora 
\cite{kisielewicz2012,kisielewicz-motyl,KISIELEWICZ20,michta2015,michta2020},
to cite but a few papers and a book. 

% starting from an unpublished paper of F.~Hiai.
% Stratonovich Michta Motyl2007
%%%
% {\color{red}Ici il faut parler de l'intégrale stochastique multivoque
% dans le cadre du calcul d'Ito et des travaux de Kisielewicz, Michta,
% Motyl...Zielona G\'ora's school \cite{KISIELEWICZ20,kisielewicz-motyl,michta}
% to cite but a few}
%\cite{malinowski-agarwal,malinowski19}
%\cite{kisielewicz20unboudedness,michta15unboundedness}

In the case of a deterministic signal $w$ with possibly infinite
variation,
Michta and Motyl \cite{MM20,MM20BGV} are the only
references so far defining a set-valued Young integral \`a la Aumann
of the form
\eqref{integral_introduction}, for convex as well as nonconvex-valued
multifunctions.  
In their approach, the set of selections $\mathcal S(F)$ is large,
namely, in the case of our setting,
$\mathcal S(F)$ is the set of all $\alpha$-H\"older continuous
selections of $F$.
Of course this is a natural definition, and the authors obtain basic
expected properties on the set-valued integral: nonemptyness,
convexity and regularity of the integral with respect to $\mathcal
S(F)$ (not $F$). In our approach, the set of selections is smaller: 
\begin{displaymath}
\mathcal S_{\alpha,r}(F) :=
\{f\textrm{ selection of }F :
\|f\|_{\alpha,T}\leqslant r\}.
\end{displaymath}
The "tuning parameter" $r > 0$ controls both the $\alpha$-H\"older
seminorm of $F$ and of the considered selections. This allows to establish
the compactness of the integral, to get the upper semicontinuity of
the integral with respect to $F$, and then to establish the existence
of solutions to some differential inclusions. Note that our integral
converges to that of Michta and Motyl \cite{MM20} when the tuning
parameter $r$ goes to $+\infty$.
However, our integral is always compact-valued,
whereas that of \cite{MM20} may be unbounded, see Example
\ref{exple:comparison} below. 

%%% Stochastic inclusions
On differential inclusions driven by $\alpha$-H\"older continuous
signals, let us mention Bailleul {\it et al.} \cite{BBC20}. In this
paper, the authors establish the existence of solutions to a
differential inclusion using the approach of Aubin and Cellina
\cite{AC84}. Let us also cite Levakov and Vas'kovskii
\cite{levakov-vaskovskii} who mix pathwise integration with respect
to fractional Brownian motion with It\^o's integral with respect to
standard Brownian motion, following Guerra and Nualart 
\cite{guerra-nualart}. These works on differential inclusions
implicitely use an Aumann type set-valued integral
of the form \eqref{integral_introduction}.
%without specifying the set  $\mathcal{S}(F)$. 

%%% 
As an application of our set-valued Young integral,
we are able to define a stochastic set-valued integral
with respect to the fractional Brownian motion (fBm) of Hurst index $H
> 1/2$, and then to establish the existence of solutions to several
types of stochastic differential inclusions driven by the fBm.
% One can refer to Kisielewicz \cite{KISIELEWICZ20} on the set-valued
% stochastic integral and stochastic differential inclusions in the
% case $H = 1/2$, deeply related to the martingale property of the
% Brownian motion, as in the point-valued framework. 

%%% LE PLAN
This paper is organized as follows.
Section \ref{section_preliminaries} recalls preliminary definitions
and results on Steiner's selections
and on the Young integral for point-valued functions.
Section \ref{section_set_valued_Young_integral} deals with the
set-valued Young integral.
Finally, Section \ref{section_fixed_points} provides a fixed-point
theorem for functionals of the set-valued Young integral which allows,
in particular, to get the existence of solutions to several types of
differential inclusions.
%

% Section : Preliminaries.

%
\section{Preliminaries}\label{section_preliminaries}
\subsection{Notations and basic definitions}
Let $d\geqslant 1$ be an integer.
\begin{enumerate}[1.]
 \item The set of nonempty closed subsets of $\R^d$ is denoted by $\cl(\R^d)$.  
 The semi-\emph{Hausdorff distance} on $\cl(\R^d)$ is denoted by $d_\Haus$:
 for all $(A,B)\in\cl(\mathbb R^d)^2$
 \begin{equation}\label{eq:dH}
 d_\Haus (A,B) =
 \max\left\{\sup_{a\in A}d(a,B)\tq
 \sup_{b\in B}d(b,A)\right\}
 =\sup_{x\in\R^d}\abs{d(x,A) - d(x,B)}
 \end{equation}
 (see, e.g., \cite{Beer}).
 For every $A\in\mathcal \cl(\R^d)$, we denote $\|A\|_{d_\Haus} :=
 d_\Haus(A,\{0_{\R^d}\}) =\sup\{\|a\|\tq a\in A\}$.

 \item Let $\cc(\R^d)$ be the space of nonempty, convex and compact
 subsets of $\mathbb R^d$. For any $C\in\cc(\R^d)$, the
 \emph{support function} of $C$ is the map
 \[
 \delta^*(.,C) :\, 
 \left\{\begin{array}{lcl}
  \R^d&\rightarrow&\R\\
  l &\mapsto&\max_{x\in C}\langle\ell,x\rangle.
 \end{array}\right.
 \]
 We have, for all $A,B\in\cc(\R^d)$,
 \[d_\Haus (A,B)
 =\sup_{\ell\in\R^d,\|\ell\|=1}\abs{\delta^*(\ell,A)-\delta^*(\ell,B)}.\]
 For any $C\in\cc(\R^d)$ and $\ell\in\mathbb R^d$, consider
 \begin{displaymath}
 Y(\ell,C) :=
 \left\{c\in C :\langle\ell,c\rangle =\delta^*(\ell,C)\right\}.
 \end{displaymath}
 If $Y(\ell,C)$ contains exactly one element, it is denoted by
 $y(\ell,C)$ and called \emph{exposed point} of $C$ with
 \emph{exposing direction} $\ell$. The set of all exposing directions
 for a given point of $C$ is denoted by $\mathcal T_C$.
 For all $A,B\in\cc(\R^d)$, the \emph{Demyanov distance} between $A$
 and $B$ is defined by
 \begin{displaymath}
 d_\D(A,B) :=
 \sup\{\|y(\ell,A) - y(\ell,B)\|\tq \ell\in\mathcal T_{A}\cap\mathcal T_{B}\}.
 \end{displaymath}
 Note that $d_\D(A,B)\geqslant d_\Haus(A,B)$ for all $A,B\in\cc(\R^d)$,
 see, e.g., \cite{rzezuchowski}.
 On $\R^d$, which we identify
 with the  subset of singletons of $\cc(\R^d)$, both distances $d_\D$
 and $d_\Haus$ 
 coincide with the distance induced by $\norm{.}$.
 
 \item For any $\alpha\in ]0,1[$,
 $\C^{\alpha\textrm{-H\"ol}}([0,T];\cc(\R^d))$
 (respectively $\C^{\alpha\textrm{-H\"ol}}_\D([0,T];\cc(\R^d))$)
 is the space of
 $\alpha$-H\"older continuous maps from $[0,T]$ into $\cc(\R^d)$,
 where $\cc(\R^d)$ is endowed with the Hausdorff distance
 (respectively, the Demyanov distance).
 In the sequel, $\C^{\alpha\textrm{-H\"ol}}([0,T];\cc(\R^d))$ is
 equipped with the $\alpha$-H\"older norm $N_{\alpha,T}(.)
 :=\|.\|_{\infty,T} +\|.\|_{\alpha,T}$, where,
 for $F\in \C^{\alpha\textrm{-H\"ol}}([0,T];$ $\cc(\R^d))$,
 \begin{displaymath}
 \|F\|_{\alpha,T}
 =\sup\left\{\frac{d_\Haus(F(s),F(t))}{|t - s|^{\alpha}}\tq
 s,t\in [0,T]\textrm{ and }s <
 t\right\},
 \end{displaymath}
 and (with inconsistent but convenient notations)
 \begin{equation*}
 \|F\|_{\infty,T}=\sup_{t\in[0,T]}\norm{F(t)}_{d_\Haus}. 
 \end{equation*}
 We denote by $\overline{B}_{\alpha,\cc}(F,\delta)$
 the closed ball of center $F$ and radius
 $\delta$ in the space $\C^{\alpha\textrm{-H\"ol}}([0,T];\cc(\R^d))$ .
 
 Similarly, the space $\C^{\alpha\textrm{-H\"ol}}_\Dem([0,T];\cc(\R^d))$ is
 equipped with the $\alpha$-H\"older norm $N_{\alpha,T,\Dem}(.)
 :=\|.\|_{\infty,T} +\|.\|_{\alpha,T,\Dem}$, where $d_\Haus$ is
 replaced by $d_\D$ in the above definition.
 
 \item For any $s,t\in [0,T]$ such that $t > s$, $\mathfrak D_{[s,t]}$ is the set of all dissections of $[s,t]$.
\end{enumerate}
%

%%%%%%%%%%%%%% Radstrom-Hormander
\begin{remark}(R{\aa}dstr\"om-H\"ormander embedding)
Since $\cc(\R^d)$ is not a vector space, 
it may seem strange to use the notation
$\|.\|_{d_\Haus}$ 
or to call
$N_{\alpha,T}$ and $N_{\alpha,T,\Dem}$ "norms". 
Actually, there exists an embedding of $\cc(\R^d)$ into a vector space
that turns %$\|.\|_{d_\Haus}$, $N_{\alpha,T}(.)$ and $N_{\alpha,T,\Dem}(.)$ 
these mappings into true
norms.

More precisely, 
let us recall briefly the R{\aa}dstr\"om-H\"ormander embedding
\cite{Radstrom,Hormander},
following R{\aa}dstr\"om's construction \cite{Radstrom}. 
The space $\cc(\R^d)$ is endowed with
the scalar multiplication $\lambda A=\{\lambda a\tq a\in A\}$
and the
addition $A+ B:={\{a+b \tq a\in A
\textrm{ and }b\in B\}}$, 
for all $A,B\in\cc(\R^d)$ and $\lambda\in\R^+$. With these operations, 
$\cc(\R^d)$ is %a convex cone
such that $A+C=B+C$ implies $A=B$
for all $A,B,C\in\cc(\R^d)$ and $\lambda_1A+\lambda_2
A=(\lambda_1+\lambda_2)A$ for all $\lambda_1,\lambda_2\in\R^+$ and
all $A\in\cc(\R^d)$. 
Taking equivalence classes of pairs $(A,B)\in\bigl(\cc(\R^d)\bigr)^2$ 
for the relation $(A,B)\sim(C,D)\Leftrightarrow
A+D=B+C$ %, as in \cite{Radstrom},
leads to the construction of a vector space $\RHspace$ which extends
the convex cone $\cc(\R^d)$,
where each element $A\in\cc(\R^d)$ is identified with
$(A,\{0_{\R^d}\})$.
The distances $d_\Haus$ and $d_\D$ are extended to $\RHspace$ by
setting
\[d_\XXX\bigl((A,B),(C,D)\bigr)=d_\XXX(A+D,B+C),\]
where $\XXX$ denotes either $\Haus$ or $\D$. 
Since $d_\Haus$ and $d_\D$ are translation invariant and positively
homogeneous, the map 
$(A,B)\mapsto d_\XXX\bigl((A,B),(\{0_{\R^d}\},\{0_{\R^d})\bigr)$  
is a norm on $\RHspace$ which induces the distance $d_\XXX$. 
Then $N_{\alpha,T}$ and $N_{\alpha,T,\Dem}$ have natural extensions as
norms on   
$\C^{\alpha\textrm{-H\"ol}}([0,T];\RHspace)$ and 
$\C^{\alpha\textrm{-H\"ol}}_\D([0,T];\RHspace)$ respectively.
%% Aumann integral wrt a nonnegative measure vs Bochner integral :
%% \cite[Theorem 4.5]{Hiai-Umegaki}
\end{remark}
%%%%%%%%%%%%

% Subsection : Generalized Steiner's selections.

%
\subsection{Steiner point and generalized Steiner selections}\label{subsection_GS_selections}
Let 
\begin{eqnarray*}
 \mathcal M & := &
 \{\mu\textrm{ probability measure on } B_{\R^d}(0,1)\\
 & &
 \hspace{3cm}\tq
 \exists\theta\in \C^1(B_{\R^d}(0,1);\mathbb R)\textrm{, }\mu(dx) =\theta(x)dx\}.
\end{eqnarray*}
%

% Definition : Steiner points.

%
\begin{definition}\label{def:Steiner point}
The \emph{Steiner point} of $C\in\cc(\R^d)$ is defined by
\begin{displaymath}
\St(C)
:=\frac{1}{v_d}\int_{\mathcal{T}_C\cap B_{\R^d}(0,1)}y(x,C)\,dx
\quad
\textrm{with}
\quad
v_d =\frac{\pi^{d/2}}{\Gamma(1 + d/2)}.
\end{displaymath}
\end{definition}
It is well known that ${\rm St}(C)\in C$ for any $C\in\cc(\R^d)$ (see the
more general Proposition \ref{Castaing_representation}
below). Furthermore, we have:
%

% Proposition : Lipschitz property Steiner points.

%
\begin{proposition}[Lipschitz property]
\label{prop:steinerlip}
The map $\St :\,\cc(\R^d)\mapsto \R^d$ is $\StLip{d}$-Lipschitz for
the Hausdorff distance $d_\Haus$, where
the sharp Lipschitz coefficient $\StLip{d}$ satisfies
\begin{equation}\label{eq:StLip}
\sqrt{\frac{2d}{\pi}}<\StLip{d}<\sqrt{\frac{2(d+1)}{\pi}}.
\end{equation}
\end{proposition}
See \cite{vitale,saint-pierre} for the exact calculation of
$\StLip{d}$.
The estimation \eqref{eq:StLip} can be found in \cite{saint-pierre}.

The Steiner point can be generalized by replacing the probability
measure $dx/v_d$ in Definition \ref{def:Steiner point} 
by any element of $\mathcal M$. This is particularly
interesting in connection with the Demyanov distance.
%

% Definition : Generalized Steiner points.

%
\begin{definition}[generalized Steiner selection]
The \emph{Generalized Steiner point} of $C\in\cc(\R^d)$, for a measure
$\mu\in\mathcal M$, is defined by
\begin{displaymath}
\St_\mu(C)
:=\int_{B_{\R^d}(0,1)}\St(Y(x,C))\mu(dx).
\end{displaymath}
The \emph{generalized Steiner selection} of a multifunction
$F : [0,T]\rightarrow\cc(\R^d)$ with respect to
a measure $\mu\in\mathcal M$, is
the map
$t\in [0,T]\mapsto\St_\mu(F(t))$.
\end{definition}
%

% Proposition : Convexity of generalized Steiner's points set.

%
\begin{proposition}\label{GS_points_convexity}
For all $C,C_1,C_2\in\cc(\R^d)$, $\mu_1,\mu_2,\mu\in\mathcal M$,
$\kappa\in [0,1]$, $\lambda,\nu\geqslant 0$, we have
\begin{align*}
 \St_{\kappa\mu_1 + (1 -\kappa)\mu_2}(C) &=
 \kappa\St_{\mu_1}(C) +(1 -\kappa)\St_{\mu_2}(C),\\
 \St_{\mu}(\lambda C_1+\nu C_2) &= 
 \lambda\St_{\mu}(C_1)+\nu \St_{\mu}(C_2).
\end{align*}
\end{proposition}
See Baier and Farkhi \cite[Lemma 4.1]{BF07} for a proof. 
%

% Theorem : Castaing representation.

%
\begin{theorem}\label{Castaing_representation}
(Castaing representation) For every measurable multifunction $F : [0,T]\rightarrow\cc(\R^d)$, there exists a sequence $(\mu_n)_{n\in\mathbb N}$ of elements of $\mathcal M$ such that, for every $t\in [0,T]$,
\begin{displaymath}
F(t) =\overline{\bigcup_{n\in\mathbb N}\{\St_{\mu_n}(F(t))\}}.
\end{displaymath}
\end{theorem}
See Dentcheva \cite[Theorem 3.4]{DENTCHEVA00} for a proof.

It is proved in \cite{DENTCHEVA00} that, for any $\mu\in\mathcal M$, 
the map $\St_{\mu} :\,\cc(\R^d)\mapsto \R^d$ is Lipschitz for
the Hausdorff distance $d_\Haus$. However there is no uniform bound
on the Lipschitz coefficient with respect to $\mu$.
But the Demyanov distance $d_\D$ can be expressed using generalized
Steiner points: 
%
 
% Proposition : Demyanov distance and generalized Steiner points.

%
\begin{proposition}[Demyanov distance and generalized Steiner points]
\label{Demyanov_Steiner}
For every $C_1,C_2\in\cc(\mathbb R^d)$,
\begin{displaymath}
d_\D(C_1,C_2) =
\sup_{\mu\in\mathcal M}\|\St_{\mu}(C_2) -\St_{\mu}(C_2)\|.
\end{displaymath}
\end{proposition}
See Baier and Farkhi \cite[Corollary 4.8]{BF07} for a proof.
%

% Subsection : Young integral for point valued functions.

%
\subsection{Young integral for single valued functions}
This subsection deals with the definition and some basic properties of
Young integral which allow to integrate a map $f\in
\C^{\alpha\textrm{-H\"ol}}([0,T];\matr)$ with respect to $w$ when
$\alpha\in ]0,1[$ and $\alpha +\beta > 1$.
Here, $\matr$ denotes the space of real matrices with $e$ rows and $d$
columns. 

Let us first introduce a compactness result which is
essential in the construction of the Young integral and the proof of
its properties, and
%an essential tool 
that we use
several times in this paper. 

%\cite[Proposition 5.28]{FV10}
% Assume (x n ) is bounded and sup n |x n | α -Höl;[0,T ] < ∞. Then x n con-
% verges  (in α <  α Hölder topology, along a subsequence) to some x ∈
% C α -Höl [0, T ] , R d .
\begin{proposition}[Compactness in Hölder spaces]\label{prop:compactnessHolder}
Let $(f_n)$ be a bounded sequence in
$\C^{(\alpha)\textrm{-H\"ol}}([0,T];\matr)$
%which is bounded for and 
such that $\sup_n\sup_{t\in[0,T]}\norm{f_n(t)}_{\matr}<+\infty$. 
Then there exists a subsequence $(f_{n_k})_{k\in\mathbb N}$ of
$(f_n)_{n\in\mathbb N}$
and $f\in\C^{(\alpha)\textrm{-H\"ol}}([0,T];\allowbreak\matr)$
such that, for every $\varepsilon\in
]0,\alpha]$, $(f_{n_k})_{k\in\mathbb N}$ converges in
$\C^{(\alpha -\varepsilon)\textrm{-H\"ol}}([0,T];\allowbreak\matr)$
to $f$. 
\end{proposition}
See Friz and Victoir \cite[Theorem 5.28]{FV10} for a proof.

% Theorem : Young integral (construction).

%
\begin{theorem}[Young integral]\label{Young_integral}
Consider $\alpha,\beta\in ]0,1]$ such that $\alpha +\beta > 1$, and two maps
\begin{displaymath}
f\in \C^{\alpha\normalfont{\textrm{-H\"ol}}}([0,T];\matr)
\textrm{ and }
w\in \C^{\beta\normalfont{\textrm{-H\"ol}}}([0,T];\mathbb R^d).
\end{displaymath}
For every $n\in\mathbb N^*$ and
$D_n=(t_{1}^{n},\dots,t_{m_n}^{n})\in\mathfrak D_{[0,T]}$
such that $\abs{D_n}\rightarrow 0$
(where $\abs{D_n}=\max_{1\leqslant k\leqslant m_n-1}(t_{k+1}^{n}-t_k^{n})$)
the limit
\begin{displaymath}
\lim_{n\rightarrow\infty}
\sum_{k = 1}^{m_n - 1}
f(t_{k}^{n})
(w(t_{k + 1}^{n}) - w(t_{k}^{n}))
\end{displaymath}
exists and does not depend on the dissection
$D_n$. 
\end{theorem}
%

% Definition : Young integral.

%
\begin{definition}[Young integral]
The limit in Theorem \ref{Young_integral} is denoted by
\begin{displaymath}
\int_{0}^{T}f(s)\,dw(s)
\end{displaymath}
and called the \emph{Young integral} of $f$ with respect to $w$ on $[0,T]$.
\end{definition}
The following theorem provides a bound on Young's integral which is
crucial in the sequel. 
%

% Theorem : Young-Love estimate.

%
\begin{theorem}[Young-Love estimate]\label{Young_Love_estimate}
Consider $\alpha,\beta\in ]0,1[$ such that $\alpha +\beta > 1$, and two maps
\begin{displaymath}
f\in \C^{\alpha\normalfont{\textrm{-H\"ol}}}([0,T];\matr)
\textrm{ and }
w\in \C^{\beta\normalfont{\textrm{-H\"ol}}}([0,T];\mathbb R^d).
\end{displaymath}
There exists a constant $\mathfrak{c}_{\alpha,\beta}\geqslant 1$,
depending only on $\alpha$ and $\beta$, such that,
for all $s,t\in[0,T]$ such that $s<t$,
\begin{displaymath}
\left\|
\int_{s}^{t}f(u)\,dw(u) - f(s)(w(t) - w(s))\right\|
\leqslant
\mathfrak c_{\alpha,\beta}\|w\|_{\beta,T}
\|f\|_{\alpha,T}|t - s|^{\alpha +\beta}.
\end{displaymath}
Therefore, for all $s,t\in[0,T]$,
\begin{displaymath}
\left\|
\int_{s}^{t}f(u)\,dw(u)\right\|
\leqslant
\mathfrak c_{\alpha,\beta}\|w\|_{\beta,T}(
\|f\|_{\alpha,T}T^{\alpha} +\|f\|_{\infty,T})|t - s|^{\beta},
\end{displaymath}
and, in particular,
\begin{equation*}
\norm{\int_{0}^{.}f(s)\,dw(s)}_{\beta,T}
\leqslant\Cte_{\alpha,\beta}\norm{w}_{\beta,T}N_{\alpha,T}(f)
\end{equation*}
with $\Cte_{\alpha,\beta,T} =\mathfrak c_{\alpha,\beta}(T^{\alpha}\vee 1)$.
\end{theorem}
See Friz and Victoir \cite[Theorem 6.8]{FV10} for a proof.
%

% Section : Set-valued Young integral.

%
\section{Set-valued Young integral}\label{section_set_valued_Young_integral}
This section deals with an Aumann-like integral based on a special
subset of selections.
We assume the following hypothesis:
\begin{itemize}
 \item[\HA]
 $F\in \C^{\alpha\textrm{-H\"ol}}([0,T];\cc(\matr))$,
 and $w\in \C^{\beta\textrm{-H\"ol}}([0,T];\mathbb R^d)$, with
 $\alpha,\beta\in]0,1]$, $\alpha +\beta>1$.
\end{itemize}
We shall sometimes consider the subcase obtained by adding
the following stronger assumption on $F$: 
\begin{itemize}
\item[\HB]
$F\in \C^{\alpha\textrm{-H\"ol}}_\Dem([0,T];\cc(\matr))$.
\end{itemize}
To show that {\HB} is not contained in {\HA}, we can consider the case when
$F(t)$ is the segment of $\R^2$ with one end at $(0;0)$ and the other
end at $(\sin t, \cos t)$ (see \cite[Example 3.3]{rzezuchowski}). Then
$d_\D(F(t),F(s))\geqslant 1$ for $s\not=t$,
thus $N_{\alpha,T,\Dem}(F)=+\infty$ for all $\alpha$, whereas $F$ is
Lipschitz for $d_\Haus$.
%

% Subsection : Special selections.

%
\subsection{Special selections}
We now define an appropriate set of selections of $F$, that will be
used to define a set-valued Young integral of $F$ with respect to
$w$.

Let us choose our "tuning
parameter'' $r$ such that
\begin{equation}
\label{eq:tuning}
r\geqslant \StLip{ed}\|F\|_{\alpha,T},
\end{equation}
where $\StLip{ed}$ is the constant
defined in Proposition \ref{prop:steinerlip}.
In the case when {\HB} is satisfied, we can alternatively take
\begin{equation}
\label{eq:tuning2}
r\geqslant \|F\|_{\alpha,T,\Dem}.
\end{equation}
Note that Condition
\eqref{eq:tuning2} can be less restrictive than \eqref{eq:tuning},
when {\HB} is satisfied.
For example, if $F(t)$ has the form $f(t)+C$, where $f$ is
single-valued and $C\in\cc(\R^d)$ is constant, we have
$\norm{F}_{\alpha,T,\Dem}=\norm{F}_{\alpha,T}\leqslant \StLip{ed}\|F\|_{\alpha,T}$.
\begin{notation}
In this section, we denote
\begin{equation}
\label{eq:rmin}
\rmin=\min\CCO{ \StLip{ed}\|F\|_{\alpha,T},
\norm{F}_{\alpha,T,\Dem} }.
\end{equation}
\end{notation}
\begin{notation}
We denote by $\mathcal S^0(F)$ the set of all measurable selections
of $F$,
\begin{align*}
 \mathcal S_{\alpha,r}(F)
 & :=\{f\in S^0(F) \tq \norm{f}_{\alpha,T}\leqslant r\}
 \textrm{ and}\\
 \mathcal S_{\St}(F)
 &:=
 \{\St_{\mu}(F(.))\tq \mu\in\mathcal M\}\subset
 \mathcal S^0(F).
\end{align*}
\end{notation}        
\begin{remark}[Dependence on $T$ of $\mathcal S_{\alpha,r}(F)$]
  \label{rem:Salpha-dependT}
  Let $t\in [0,T[$. 
  If $f$ is an $\alpha$-Hölder selection of $F$ on $[0,t]$,
  with  $\norm{f}_{\alpha,t}\leqslant r$, there does not necesarily
  exists a selection $\overline{f}\in\mathcal S_{\alpha,r}(F)$ wich extends
  $f$ on $[0,T]$ and such that $\norm{\overline{f}}_{\alpha,T}\leqslant r$. 
  So, $\mathcal S_{\alpha,r}(F)$ depends on $T$, but we chose to keep
  the relatively light notation $\mathcal S_{\alpha,r}(F)$
  without stressing this fact.

 %% {\color{red} What if $r\geqslant\|F\|_{\alpha,T,\Dem} $?}
\end{remark}

%

% Proposition : Basic properties of the selections set.

%
\begin{proposition}\label{prop:basic}
For every $r\geqslant \rmin$,
the set $\mathcal S_{\alpha,r}(F)$ is nonempty and
convex.
\end{proposition}
%

% Proof.

%
\begin{proof}
The convexity of $\mathcal S_{\alpha,r}(F)$ stems from the convexity
of the norm $N_{\alpha,T}$. 

If \eqref{eq:tuning} is satisfied, $\mathcal S_{\alpha,r}(F)$ is
nonempty because it contains $\St(F)$
by Proposition \ref{prop:steinerlip}.
Indeed, we have, for $s,t\in[0,T]$,
\begin{equation*}
\|\St(F(t)) -\St(F(s))\| \leqslant 
\StLip{ed} d_\Haus(F(t),F(s))\leqslant \StLip{ed} \|F\|_{\alpha,T}|t - s|^{\alpha},
\end{equation*}
thus $\norm{\St(F)}_{\alpha,T}\leqslant r$.

Similarly, if {\HB} and \eqref{eq:tuning2} are satisfied, we have $\St(F)\in
\mathcal S_{\alpha,r}(F)$
by Proposition \ref{Demyanov_Steiner} and since, in that case, 
$N_{\alpha,T,\Dem}(F)<\infty$.
\end{proof}
%

% Proposition : Basic properties of the GS selections set - A.

%
\begin{remark}[Basic properties of the selections sets]
\label{basic_properties_GS_selections-A}
\rule{1em}{0em}
\begin{enumerate}
 \item  The set $\mathcal S_{\St}(F)$ is a nonempty and convex subset of
 $C^{\alpha\normalfont{\textrm{-H\"ol}}}([0,T];\matr)$
 such that, for every $f\in\mathcal S_{\St}(F)$, $N_{\alpha,T}(f)\leqslant
 N_{\alpha,T,\Dem}(F)$.
 Moreover, there exists a sequence $(f_n)_{n\in\mathbb N}$ of elements of $\mathcal S_{\mathrm{St}}(F)$ such that, for every $t\in [0,T]$,
 \begin{displaymath}
 F(t) =
 \overline{\bigcup_{n\in\mathbb N}\{f_n(t)\}}.
 \end{displaymath}
 Indeed, the convexity of $\mathcal S_{\St}(F)$ follows from
 Proposition  \ref{GS_points_convexity},
 the Castaing representation from Proposition
 \ref{Castaing_representation},
 and the comparison with  $N_{\alpha,T,\Dem}(F)$ stems from 
 Proposition \ref{Demyanov_Steiner}.
 
 \item Consequently, if {\HB} and \eqref{eq:tuning2} are satisfied, we have $\mathcal
   S_{\St}(F)\subset\mathcal S_{\alpha,r}(F)$
   (since in this case $N_{\alpha,T,\Dem}(F)<\infty$),
 and there exists a sequence $(f_n)_{n\in\mathbb N}$ of elements of
 $\mathcal S_{\alpha,r}(F)$ such that, 
 for every $t\in [0,T]$,
 \begin{equation}\label{basic_properties_selections_set_1}
 F(t) =
 \overline{\bigcup_{n\in\mathbb N}\{f_n(t)\}}.
 \end{equation}
 \end{enumerate}
\end{remark}
Let us establish some topological properties of $\mathcal S_{\alpha,r}(F)$ in $\mathbb L^2([0,T];\matr)$ and in $\C^{\alpha\textrm{-H\"ol}}([0,T];\matr)$.
%

% Proposition : Basic topological properties of the selections set in \C^{\alpha\textrm{-H\"ol}}.

%
\begin{proposition}\label{basic_topological_properties_selections_set_Holder}
For every $r\geqslant \rmin$, 
the set $\mathcal S_{\alpha,r}(F)$
is a bounded and closed subset of $\C^{\alpha\normalfont{\textrm{-H\"ol}}}([0,T];\matr)$. Moreover, for every sequence $(f_n)_{n\in\mathbb N}$ of elements of $\mathcal S_{\alpha,r}(F)$, there exists a subsequence $(f_{n_k})_{k\in\mathbb N}$ of $(f_n)_{n\in\mathbb N}$ such that, for every $\varepsilon\in ]0,\alpha]$, $(f_{n_k})_{k\in\mathbb N}$ converges in $\C^{(\alpha -\varepsilon)\normalfont{\textrm{-H\"ol}}}([0,T];\matr)$ to an element of $\mathcal S_{\alpha,r}(F)$.
\end{proposition}
%

% Proof.

%
\begin{proof}
Let $(f_n)_{n\in\mathbb N}$ be a sequence of elements of $\mathcal S_{\alpha,r}(F)$. By the definition of $\mathcal S_{\alpha,r}(F)$,
\begin{displaymath}
\sup_{n\in\mathbb N}\|f_n\|_{\infty,T}\leqslant \norm{F}_{\infty,T}
\textrm{ and }
\sup_{n\in\mathbb N}\|f_n\|_{\alpha,T}
\leqslant r.
\end{displaymath}
Therefore, %by Friz and Victoir \cite[Proposition 5.28]{FV10},
by Proposition \ref{prop:compactnessHolder},
there exists a subsequence $(f_{n_k})_{k\in\mathbb N}$ of $(f_n)_{n\in\mathbb N}$ such that, for every $\varepsilon\in ]0,\alpha]$, $(f_{n_k})_{k\in\mathbb N}$ converges in $\C^{(\alpha -\varepsilon)\textrm{-H\"ol}}([0,T];\matr)$ to an element of $\mathcal S_{\alpha,r}(F)$.
\end{proof}
%

% Proposition : Topological properties of the selections set in L^2.

%
\begin{proposition}\label{topological_properties_selections_set_L2}
$\mathcal S_{\alpha,r}(F)$ is a bounded, closed and sequentially compact subset of $\mathbb L^2([0,T];\matr)$.
\end{proposition}
%

% Proof.

%
\begin{proof}
Let  $(f_n)_{n\in\mathbb N}$ be a sequence of elements of $\mathcal
S_{\alpha,r}(F).$
According to
Proposition \ref{basic_topological_properties_selections_set_Holder},
there exists a subsequence $(f_{n_k})_{k\in\mathbb N}$ of $(f_n)_{n\in\mathbb N}$ such that, for every $\varepsilon\in ]0,\alpha]$, $(f_{n_k})_{k\in\mathbb N}$ converges in $\C^{(\alpha -\varepsilon)\normalfont{\textrm{-H\"ol}}}([0,T];\matr)$ to an element of $\mathcal S_{\alpha,r}(F)$. 
Note that
\begin{align}\label{maj-ui}
\sup_k \|f_{n_k}\|_{\infty,T} \leqslant \|F\|_{\infty,T}<\infty.
\end{align}
Then, the sequence $(f_{n_k})_k$ is uniformly integrable with respect to the Lebesgue measure on $[0,T]$ 
and $(f_{n_k})_k$ converges to $f$ in $\mathbb L^1([0,T];\matr).$ Using estimation \eqref{maj-ui}, the convergence holds in $\mathbb L^2([0,T];\matr).$ Thus, $\mathcal S_{\alpha,r}(F)$ is a {sequentially} compact subset of $\mathbb L^2([0,T];\matr)$.
\end{proof}
%

% Subsection : Aumann-Young integral.

%
\subsection{Aumann-Young integral}
Now, consider $w\in \C^{\beta\textrm{-H\"ol}}([0,T];\mathbb R^d)$ and
let us define a set-valued Young integral with respect to $w$, using
special sets of selections.
%

% Definition : Set-valued Young integral.

%
\begin{definition}[Aumann-Young integral]
\label{set_valued_Young_integral}
The \emph{Aumann-Young integral} of $F$ with respect to $w$ and parameters
$\alpha$ and $r$ is defined by
\begin{displaymath}
\Aint{\alpha,r}_{0}^{T}F(t)\,dw(t) =
J_{T,\alpha,r}(F,w) :=
\left\{\int_{0}^{T}f(t)\,dw(t)
\tq f\in\mathcal S_{\alpha,r}(F)\right\}.
\end{displaymath}
\end{definition}
\begin{remark}\label{rem:michta-motyl}
Michta and Motyl define a larger Aumann-Young integral in
\cite{MM20}. For convex-valued $F$, their integral is
\begin{multline*}
 \Aint{\alpha,+\infty}_{0,+\infty}^{T}F(t)\,dw(t)
 = J_{T,\alpha,+\infty}(F,w) \\
 \begin{aligned}
 &:=
 \left\{ \int_{0}^{T}f(t)\,dw(t)
 \tq f\in\mathcal S_{0}(F)\cap
 \C^{\alpha\textrm{-H\"ol}}([0,T];\matr) \right\}\\
 &=\bigcup_{r\geqslant 0}J_{T,\alpha,r}(F,w).
 \end{aligned}
\end{multline*}
See Example \ref{exple:comparison} below for a comparison with
$\Aint{\alpha,r}_{0,+\infty}^{T}F(t)\,dw(t)$
when $r<\infty$.
\end{remark}
The following proposition provides some basic
properties of the set-valued Aumann-Young integral.
%

% Proposition : Basic properties of the set-valued Young integral.

%
\begin{proposition}\label{basic_properties_set_valued_Young_integral}
For every $r\geqslant \rmin$,
the Aumann-Young integral of $F$ with respect to $w$ and parameters $\alpha$ and $r$ is a nonempty, bounded, closed and convex subset of $\R^e$.
\end{proposition}
%

% Proof.

%
\begin{proof}
Let us prove each property of the Aumann-Young integral of $F$ with respect to $w$ stated in Proposition \ref{basic_properties_set_valued_Young_integral}.

Since $\mathcal S_{\alpha,r}(F)$ is nonempty (resp.~convex), $J_{T,\alpha,r}(F,w)$ is nonempty (resp.~convex).
   
By Theorem \ref{Young_Love_estimate}, for every $f\in\mathcal S_{\alpha,r}(F)$,
\begin{align*}
 \left\|\int_{0}^{T}f(t)\,dw(t)\right\|
 \leqslant &
 \mathfrak c_{\alpha,\beta}T^{\beta}\|w\|_{\beta,T}(\|f\|_{\alpha,T}T^{\alpha} +\|f\|_{\infty,T}\\
 \leqslant &
 \mathfrak c_{\alpha,\beta}T^{\beta}(T^{\alpha}\vee 1)\|w\|_{\beta,T}(r+\norm{F}_{\infty,T}).
\end{align*}
Then, the Aumann-Young integral of $F$ with respect to $w$ is a bounded subset of $\R^e$.

Consider a converging sequence $(j_n)_{n\in\mathbb N}$ of elements of $J_{T,\alpha,r}(F,w)$. Its limit is denoted by $j$. By the definition of $J_{T,\alpha,r}(F,w)$, for every $n\in\mathbb N$, there exists $f_n\in\mathcal S_{\alpha,r}(F)$ such that
\begin{displaymath}
j_n =\int_{0}^{T}f_n(t)\,dw(t).
\end{displaymath}
By Proposition \ref{basic_topological_properties_selections_set_Holder}, there exists a subsequence $(f_{n_k})_{k\in\mathbb N}$ of $(f_n)_{n\in\mathbb N}$ such that, for any $\varepsilon\in ]0,\alpha]$, $(f_{n_k})_{k\in\mathbb N}$ converges in $\C^{(\alpha -\varepsilon)\textrm{-H\"ol}}([0,T];\matr)$ to an element $f$ of $\mathcal S_{\alpha,r}(F)$. So, by Theorem \ref{Young_Love_estimate}, for $\varepsilon <\alpha +\beta -1,$ for every $k\in\mathbb N$,
\begin{multline*}
 \left\|j_{n_k} -\int_{0}^{T}f(t)\,dw(t)\right\|\\
 \begin{aligned}
 \leqslant &
 \mathfrak c_{\alpha,\beta}\|w\|_{\beta,T}T^{\beta}(\|f_{n_k} - f\|_{\alpha -\varepsilon,T}T^{\alpha -\varepsilon} +
 \|f_{n_k} - f\|_{\infty,T})\\
 \leqslant &
 \mathfrak c_{\alpha,\beta}\|w\|_{\beta,T}T^{\beta}(T^{\alpha -\varepsilon}\vee 1)N_{\alpha -\varepsilon,T}(f_{n_k} - f)
 \xrightarrow[k\rightarrow\infty]{} 0.
 \end{aligned}
\end{multline*}
Therefore,
\begin{displaymath}
j =\int_{0}^{T}f(t)\,dw(t),
\end{displaymath}
and then $J_{T,\alpha,r}(F,w)$ is a closed subset of $\R^e$.
\end{proof}
%

% Example : Comparison with Aumann-Michta-Motyl.

%
\begin{example}[Comparison with Michta and Motyl's Young integral]
 \label{exple:comparison}
 In this example, Michta and Motyl's Young integral \cite{MM20} is
 equal to $\R$, whereas, by
 Proposition \ref{basic_properties_set_valued_Young_integral},
 $J_{T,\alpha,r}(F,w)$ is a compact interval.
 Assume that $\beta \leqslant \frac{1}{2},$ and
 let
 \[w(t)=t^{2\beta}\cos(\pi/t)\text{ and }F(t)=[-1,1], \quad
 (t\in[0,1]).\]
 We have
 $N_{\alpha,T}(F)=N_{\alpha,T,\Dem}(F)=1$
 and, for $0\leqslant s<t\leqslant 1$, 
 \begin{align*}
 \abs{w(t)-w(s)}
 \leqslant& \abs{\CCO{t^{2\beta} -s^{2\beta}}\cos (\pi /t)}
 +\abs{s^{2\beta}(\cos( \pi/ t) -\cos( \pi/s) )}\\
 \leqslant& \CCO{t^{2\beta} -s^{2\beta}}\\
 &+s^{2\beta}\abs{ \cos( \pi/ t) -\cos( \pi/s) }^{1-\beta}
 \abs{ \cos( \pi/ t) -\cos( \pi/s) }^{\beta}\\
 \leqslant& \CCO{t^{2\beta} -s^{2\beta}}
 +2^{1-\beta}\pi^{\beta}s^{2\beta}\CCO{\frac{1}{s^{\beta}}-\frac{1}{t^{\beta}}}\\
 \leqslant & 2\CCO{t^{\beta} -s^{\beta}}
 +2^{1-\beta}\pi^{\beta}\CCO{t^{\beta} -s^{\beta}}\\
 \leqslant & \CCO{t-s}^{\beta}\CCO{2+2^{1-\beta}\pi^{\beta}},
\end{align*}
thus $\|w\|_{\beta,1}\leqslant 2+2^{1-\beta}\pi^{\beta}$.

Let us define a sequence $(f_n)_{n\geqslant 1}$ of selections of $F$ by
\[
f_n(t)=\left\{
\begin{array}{lcl}
\sin(\pi/t)&\text{ if }&\frac{1}{n}\leqslant t\leqslant 1\\
0&\text{ if }&0\leqslant t\leqslant\frac{1}{n}.
\end{array}\right.
\]
Clearly, $\norm{f_n}_{\alpha,1}\rightarrow+\infty$ when $n\rightarrow\infty$.
Furthermore,
\begin{align*}
 \int_0^1f_n(t)\,dw(t)
 &=\int_{1/n}^1\sin(\pi/t)\,d\bigr(t^{2\beta}\cos(\pi/t)\bigl)\\
 &=\int_{1/n}^1\sin(\pi/t)
 \Bigr(2\beta t^{2\beta-1}\cos(\pi/t)-\pi t^{2\beta-2}\sin(\pi/t)\Bigl)\,dt\\
 &=\int_1^n\sin(\pi u)
 \Bigr(2\beta u^{1-2\beta}\cos(\pi u)-\pi u^{2-2\beta}\sin(\pi
 u)\Bigl)
 \frac{1}{u^2}\,du\\
 &\leqslant 2\beta\int_{1}^nu^{-1-2\beta}du
 -2\sum_{k=1}^{n-1}(k+1)^{-2\beta}\int_{k}^{k+1}\sin^2(\pi
 u)\,du\\
 &\rightarrow -\infty\text{ when }n\rightarrow\infty.
\end{align*}
By convexity of the Aumann integral and symmetry of $F$,
this shows that the integral of
$F$ with respect to $w$ in
the sense of Michta and Motyl \cite{MM20} is the whole line $\R$.
\end{example}
%

%  Proposition : Lipschitz continuity with respect to the driving signal.

%
\begin{proposition}[Lipschitz continuity result with respect to the driving signal]
\label{continuity_set_valued_Young_integral_with_respect_w}
For every $r\geqslant \rmin$, 
the set-valued map
\begin{displaymath}
J_{T,\alpha,r}(F,.) :
\left\{
\begin{array}{rcl}
 \C^{\beta\normalfont{\textrm{-H\"ol}}}([0,T];\mathbb R^d)
 & \longrightarrow & 
 \cc(\mathbb R^e)\\
 w & \longmapsto & J_{T,\alpha,r}(F,w)
\end{array}\right.
\end{displaymath}
is Lipschitz continuous when $\cc(\mathbb R^e)$ is endowed with the Hausdorff distance $d_\Haus$.
\end{proposition}
%

% Proof.

%
\begin{proof}
Consider $w^1,w^2\in \C^{\beta\textrm{-H\"ol}}([0,T];\mathbb R^d)$ and $j^1\in J_{T,\alpha,r}(F,w^1)$. So, there exists $f^1\in\mathcal S_{\alpha,r}(F)$ such that
\begin{displaymath}
j^1 =\int_{0}^{T}f^1(s)\,dw^1(s),
\end{displaymath}
and by Theorem \ref{Young_Love_estimate},
\begin{multline*}
 \begin{aligned}
 &
 d(j^1,J_{T,\alpha,r}(F,w^2))
 =\inf_{f\in\mathcal S_{\alpha,r}(F)}
 \left\|\int_{0}^{T}f(t)\,dw^2(t) -\int_{0}^{T}f^1(t)\,dw^1(t)\right\|\\
 \leqslant &
 \left\|\int_{0}^{T}f^1(t)d(w^2 - w^1)(t)\right\| +
 \inf_{f\in\mathcal S_{\alpha,r}(F)}
 \left\|\int_{0}^{T}(f - f^1)(t)\,dw^2(t)\right\|\\
 \leqslant &
 \mathfrak c_{\alpha,\beta}T^{\beta}(T^{\alpha}\vee 1)\left[N_{\alpha,T}(f^1)\|w^1 - w^2\|_{\beta,T} +
 \|w^2\|_{\beta,T}\inf_{f\in\mathcal S_{\alpha,r}(F)}N_{\alpha,T}(f - f^1)
 \right].
\end{aligned}
\end{multline*}
Since $f^1\in\mathcal S_{\alpha,r}(F)$,
$N_{\alpha,T}(f^1)\leqslant r+\norm{F}_{\infty,T}$
and the second term in the right-hand
side of the previous inequality is null.
Then
\begin{displaymath}
d(j^1,J_{T,\alpha,r}(F,w^2))
\leqslant
\mathfrak c_{\alpha,\beta}T^{\beta}(T^{\alpha}\vee 1)
(r +\norm{F}_{\infty,T})\|w^1 - w^2\|_{\beta,T}
\end{displaymath}
and, by symmetry,
\begin{displaymath}
d(j^2,J_{T,\alpha,r}(F,w^1))
\leqslant
\mathfrak c_{\alpha,\beta}T^{\beta}(T^{\alpha}\vee 1)
(r +\norm{F}_{\infty,T})\|w^1 - w^2\|_{\beta,T}
\end{displaymath}
for every $j^2\in J_{T,\alpha,r}(F,w^2)$. Therefore,
\begin{multline*}
 d_\Haus(J_{T,\alpha,r}(F,w^1),J_{T,\alpha,r}(F,w^2))\\
 \begin{aligned}
 & =\max\left\{\sup_{j^1\in
 J_{T,\alpha,r}(F,w^1)}d(j^1,J_{T,\alpha,r}(F,w^2))
 \tq
 \sup_{j^2\in J_{T,\alpha,r}(F,w^2)}d(j^2,J_{T,\alpha,r}(F,w^1))\right\}\\
 & \leqslant
 \mathfrak c_{\alpha,\beta}T^{\beta}(T^{\alpha}\vee 1)
 (r+\norm{F}_{\infty,T})\|w^1 - w^2\|_{\beta,T}.
\end{aligned}
\end{multline*}
\end{proof}
The following proposition provides
semicontinuity results for the set-valued Young's integral.
Let us first recall the topological superior and inferior limits in
Kuratowski's sense 
for a sequence of sets, see,
e.g., \cite{Beer}.
If $(A_n)$ is a sequence of closed subsets of a metric space $\M$,
let us denote
\begin{itemize}
 \item[] $\Li_{n\rightarrow\infty}{A_n}$ { the set of limits of
 sequences $(x_n)$ such that $x_n\in A_n$ for every $n$,}\\
 \item[] $\Ls_{n\rightarrow\infty}{A_n}$ { the set of limits of sequences $(x_{n})$ such that $x_n\in A_{m_n}$ for every $n$ for some
 subsequence $(A_{m_n})$ of $(A_n)$.}
\end{itemize}
Clearly,
$\Li_{n\rightarrow\infty}{A_n}\subset\Ls_{n\rightarrow\infty}{A_n}$.  
We say that $(A_n)$ \emph{converges in Kuratowski's sense}
to $A\subset \M$ if
\[\Ls_{n\rightarrow\infty}{A_n}\subset A\subset
\Li_{n\rightarrow\infty}{A_n}.\]
Convergence in
Kuratowski's sense is weaker than convergence for the Hausdorff
distance,
however both convergences are equivalent if $\M$ is compact. 
Indeed, 
by, e.g., \cite[Theorem 3.1 page 51]{Christensen},
the set $\komp(\M)$ of compact subsets of $\M$ is compact for the
topology of Hausdorff
distance, thus this topology coincide with any weaker separated (T2)
topology on $\komp(\M)$. But, if $\M$ is compact, the convergence in 
Kuratowski's sense is associated with a separated topology
(see, e.g., \cite[Theorem 5.2.6]{Beer}).
%

% Proposition : Continuity of the set-valued Young integral.

%
\begin{proposition}[Semicontinuity with respect to $F$]
\label{continuity_set_valued_Young_integral}
 Let $r\geqslant \rmin. $
 \begin{enumerate}[1.]
 \item Let $(F_n)_{n\in\mathbb N}$ be a sequence of elements of
 $\C^{\alpha\normalfont{\textrm{-H\"ol}}}([0,T];$ $\ccm)$ such that 
 \begin{displaymath}
 \Ls_{n\rightarrow\infty}F_n(t) \subset F(t)
 \tq
 \forall t\in [0,T],
 \end{displaymath}
 and that
 \[F_n\in\overline
 B_{\alpha,\ccm}(0,r +\varepsilon_n)\]
 for every $n\in\N$, for some
 sequence $(\varepsilon_n)_{n\in\mathbb N}$ of elements of
 $\R_+$, converging to $0$.
 Then
 \begin{displaymath}
 \Ls_{n\rightarrow\infty}
 J_{T,\alpha,r +\varepsilon_n}(F_n,w)\subset J_{T,\alpha,r}(F,w).
 \end{displaymath}
 \item Assume furthermore that
 \[F(t)\subset\Li_{n\rightarrow\infty}F_n(t)\tq
 \forall t\in [0,T].\]
 Then
 \[J_{T,\alpha,r}(F,w)\subset \Li_{n\rightarrow\infty} J_{T,\alpha,2r+\epsilon_n}(F_n,w).
 \]
\end{enumerate}
\end{proposition}
%

% Proof.

%
\begin{proof}
1. Let us prove that
\begin{displaymath}
\overline J :=\Ls_{n\rightarrow\infty}J_{T,\alpha,r +\varepsilon_n}(F_n,w)
\subset J_{T,\alpha,r}(F,w).
\end{displaymath}
Consider $j\in\overline J$. Then, there exists a sequence $(f_n)_{n\in\mathbb N}$ of elements of the space $\C^{\alpha\textrm{-H\"ol}}([0,T];\matr)$ such that
\begin{displaymath}
f_n\in\mathcal S_{\alpha,r +\varepsilon_n}(F_n)
\tq
\forall n\in\mathbb N
\end{displaymath}
and
\begin{displaymath}
j =\lim_{k\rightarrow\infty}\int_{0}^{T}f_{n_k}(t)\,dw(t),
\end{displaymath}
where $(f_{n_k})_{k\in\mathbb N}$ is a subsequence of $(f_n)_{n\in\mathbb N}$.
By the definition of $\mathcal S_{\alpha,r +\varepsilon_n}(F_n)$, $n\in\mathbb N$,
\begin{displaymath}
\|f_{n_k}(t)\|\leqslant
\sup_{n,s}\|F_n(s)\|_{d_\Haus}
\leqslant
\norm{F}_{\infty,T} +\sup_{n\in\mathbb N}\varepsilon_n <\infty
\end{displaymath}
for every $k\in\mathbb N$ and $t\in [0,T]$, and
\begin{displaymath}
\sup_{k\in\mathbb N}\|f_{n_k}\|_{\alpha,T}
\leqslant r +\sup_{n\in\mathbb N}\varepsilon_n.
\end{displaymath}
Then, %by Friz and Victoir \cite[Proposition 5.28]{FV10},
by Proposition \ref{prop:compactnessHolder},
there exists a subsequence $(f_{m_k})_{k\in\mathbb N}$ of $(f_{n_k})_{k\in\mathbb N}$ such that, for every $\varepsilon\in ]0,\alpha]$, such that $\varepsilon <\alpha+\beta -1,$ $(f_{m_k})_{k\in\mathbb N}$ converges in $\C^{(\alpha -\varepsilon)\textrm{-H\"ol}}([0,T];\mathbb R^d)$ to an element $f$ of $\C^{\alpha\textrm{-H\"ol}}([0,T];\mathbb R^d)$. So,
\begin{displaymath}
j =\int_{0}^{T}f(t)\,dw(t).
\end{displaymath}
It remains to check that $f\in\mathcal S_{\alpha,r}(F)$. For any $t\in [0,T]$, since $f_{m_k}(t)\in F_{m_k}(t)$ for every $k\in\mathbb N$, and since $f$ is in particular the pointwise limit of $(f_{m_k})_{k\in\mathbb N}$,
\begin{displaymath}
f(t) =\limsup_{k\rightarrow\infty}f_{m_k}(t)\in\Ls_{k\rightarrow\infty}F_{m_k}(t) = F(t).
\end{displaymath}
Moreover, by Friz and Victoir \cite[Lemma 5.12]{FV10},
\begin{displaymath}
\norm{f}_{\alpha,T}
\leqslant
\liminf_{k\rightarrow\infty}
\norm{f_{m_k}}_{\alpha,T}
\leqslant r +\lim_{n\rightarrow\infty}\varepsilon_n = r.
\end{displaymath}
Therefore, $j\in J_{T,\alpha,r}(F,w)$.

2. The supplementary hypothesis implies that $(F_n(t))$ converges in
Kuratowski's sense to $F(t)$.
Furthermore, since $(F_n)$ is bounded in
$\C^{\alpha\normalfont{\textrm{-H\"ol}}}([0,T];$ $\ccm)$, it is
bounded for $\norm{.}_{\infty,T}$,
thus
$(F_n(t))$ converges to $F(t)$ for the Hausdorff distance.

Let $j\in J_{T,\alpha,r}(F,w)$, and let $f\in\mathcal
S_{\alpha,r}(F)$ such that $j =\int_{0}^{T}f(t)\,dw(t)$.
Set $f_n(t)=\proj{F_n(t)}(f(t))$ for every $t\in[0,T]$ and each
integer $n$, where $\proj{F_n(t)}$ denotes the orthogonal projection
on $F_n(t)$. We have, for any $t\in[0,T]$,
\[\norm{f(t) - f_n(t)}\leqslant d_\Haus(F(t),F_n(t))\rightarrow 0\text{ when }
n\rightarrow\infty,\]
thus $(f_n)$ converges uniformly to $f$. 
Furthermore, for any $n$ and for $s,t\in[0,T]$,
\begin{align*}
 \norm{f_n(t)-f_n(s)}
 &\leqslant \norm{\proj{F_n(t)}(f(t))-\proj{F_n(t)}(f(s))}
 +\norm{\proj{F_n(t)}(f(s))-\proj{F_n(s)}(f(s))}\\
 &\leqslant \norm{f(t)-f(s)}+ d_\Haus(F_n(t),F_n(s)). 
\end{align*}
Indeed, it is well known that the projection operator
$\proj{F_n(t)}$ is non expansive (see, e.g., \cite[page 118]{urruty}),
and the estimation of
$\norm{\proj{F_n(t)}(f(s)) -\proj{F_n(s)}(f(s))}$ follows from \eqref{eq:dH}.
Since $r\geqslant\norm{F}_{\alpha,T}$, we deduce
\begin{equation*}
\norm{f_n}_{\alpha,T}\leqslant\norm{f}_{\alpha,T} +\norm{F_n}_{\alpha,T}
\leqslant 2r+\epsilon_n.
\end{equation*}
Thus $j_n :=\int_{0}^{T}f_n(t)\,dw(t)\in J_{T,\alpha,2r +\epsilon_n}(F_n,w)$. 
Furthermore, thanks to \cite[Proposition 6.12]{FV10}, we have 
$\lim_{n\rightarrow\infty}j_n = j$.
\end{proof}
The preceding result can be improved when the multifunctions $F_n$ are
constructed from $F$ using some recipe which can also be applied to
their selections. This can be useful for numerical approximations.
%

% Corollary : Time discretization.

%
\begin{proposition}\label{cor:approx}(Time discretization of the
  multivalued integral)
  Let $(D_n)$ be a sequence of dissections of $[0,T]$, say,
  $D_n=(t_0^n,\dots,t_{m_n}^n)$, $0=t_0<\dots<t_{m_n}^n=T$, and
  assume that
  $\abs{D_n}$ converges to $0$, where  $\abs{D_n}=\max_{1\leqslant
  i\leqslant m_n-1}(t_{i+1}^n-t_i^n)$ is the
  mesh of $D_n$.
  For each $n$, for each $i\in\{1,\dots,m_n-1\}$ and for any
  $t\in[t_i^n,t_{i+1}^n]$, set
  \[
  F_n(t)=\frac{t-t_i^n}{t_{i+1}^n-t_i^n}F(t_i^n)+\frac{t_{i+1}^n-t}{t_{i+1}^n-t_i^n}F(t_{i+1}^n).
  \]
  Then 
  \[ \lim_{n\rightarrow\infty} d_\Haus\left(J_{T,r}(F_n,w),J_{T,r}(F,w) \right)=0.\]
\end{proposition}
%

% Proof.

%
\begin{proof}
By uniform continuity of $F$ on $[0,T]$, the sequence $(F_n)$
converges uniformly to $F$ for $d_\Haus$, and $\norm{F_n}_{d_\Haus}\leqslant
\norm{F}_{d_\Haus}$ for all $n$.
Furthermore, we have $N_{\alpha,T}(F_n)\leqslant N_{\alpha,T}(F)$ for all
$n$, see \cite{cox}.

From Part 1 of Proposition \ref{continuity_set_valued_Young_integral},
we have that
$\Ls J_{T,r}(F_n,w)\subset J_{T,\alpha,r}(F,w)$.

Now, let $j\in J_{T,\alpha,r}(F,w)$, and let $f\in \mathcal
S_{\alpha,r}(F)$ such that $j=\int_{0}^{T}f(t)\,dw(t)$.
Define $f_n\in\mathcal S_{\alpha,r}(F_n)$ by
\[f_n(t)=\frac{t-t_i^n}{t_{i+1}^n-t_i^n}f(t_i^n)+\frac{t_{i+1}^n-t}{t_{i+1}^n-t_i^n}f(t_{i+1}^n).
\]
for each $i\in\{1,\dots,m_n-1\}$ and for every
$t\in[t_i^n,t_{i+1}^n]$.
Then $(f_n)$ converges uniformly to $f$, and we conclude as in the
proof of Proposition \ref{continuity_set_valued_Young_integral} that
$J_{T,\alpha,r}(F,w)\subset \Li J_{T,r}(F_n,w)$, thus
$({J_{T,r}(F_n,w)})_n$ converges to ${J_{T,r}(F_n,w)}$ in Kuratowski's
sense.

Since the sequence $({J_{T,r}(F_n,w)})_n$ is bounded for
$\norm{.}_{d_\Haus}$, it is relatively compact for the Hausdorff distance, 
we deduce that it converges to ${J_{T,r}(F_n,w)}$ for $d_\Haus$. 
\end{proof}
Let us conclude this section by investigating the indefinite Aumann-Young integral
\begin{align*}
 t\mapsto \Aint{\alpha,r}_{0}^{t}F(s)\,dw(s)
 &=
 J_{t,\alpha,r}(F,w) \\
 &:=
 \left\{\int_{0}^{T}f(s)\mathbf{1}_{[0,t]}(s)\,dw(s)
 \tq f\in\mathcal S_{\alpha,r}(F)\right\}.
\end{align*}
%
% {\color{red}La notation $\mathcal S_{\alpha,T,r}(F)$ serait sans doute
% plus appropriée car l'ensemble des sélection dépend aussi de T (on
% peut prolonger à $[0,T]$ une sélection $\alpha$-H\"older définie sur
% $[0,t]$, mais la norme $\alpha$-H\"older du prolongement ne reste
% pas nécessairement en dessous de $r$).
% Les notations $\Aint{\alpha,r}_{0}^{t}F(s)\,dw(s)$ et
% $J_{t,\alpha,r}(F,w)$ pr\^etent \'egalement \`a confusion. Il
% faudrait écrire par exemple $J_{t,\alpha,T, r}(F,w)$. Je propose d'en
% faire seulement une remarque, sinon les notations vont devenir
% insupportablement compliquées.}

%%%%%%%%%%%%%%%%% see Remark \ref{rem:Salpha-dependT}
\begin{remark}[Dependence on $T$ of the indefinite integral]
  \label{rem:dependT}
  Since $\Aint{\alpha,r}_{0}^{t}F(s)\,dw(s)$ is built using
  elements of $\mathcal S_{\alpha,r}(F)$,
  it follows from Remark \ref{rem:Salpha-dependT} that 
  our indefinite integral depends on $T$.
  More accurate but rather heavy notations 
  could be
  $\Aint{\alpha,T,r}_{0}^{t}F(s)\,dw(s)
  =
  J_{t,\alpha,T,r}(F,w)$. 
\end{remark}

%%%%%%%%%%%%%%%%%%%%%%%%%%%%
\begin{remark}\label{rem-ppte-T}
The previous results on $J_{T,\alpha,r}(F,w)$ remain true for
$J_{t,\alpha,r}(F,w)$, by the same arguments.
\end{remark}

%

% Proposition : Hölder continuity of the Aumann-Young_integral.

%
\begin{proposition}[Continuity of the indefinite Aumann-Young integral]
\label{Holder_continuity_Aumann_Young_integral}
The set-valued map $t\in [0,T]\mapsto J_{t,\alpha,r}(F,w)$ is
$\beta$-H\"older continuous with constant
$\Cte_{\alpha,\beta,T}(\norm{F}_{\infty,T}+r)\|w\|_{\beta,T}$
when $\cc(\R^e)$ is endowed with the Hausdorff distance
$d_\Haus$,
where $\Cte_{\alpha,\beta,T}$ is the constant defined in 
Theorem \ref{Young_Love_estimate}.
\end{proposition}
%

% Proof.

%
\begin{proof}
Consider $s,t\in [0,T]$ with $s < t$, and $j_t\in J_{t,\alpha,r}(F,w)$. So, there exists $f_t\in\mathcal S_{\alpha,r}(F)$ such that
\begin{displaymath}
j_t =\int_{0}^{t}f_t(u)\,dw(u),
\end{displaymath}
and by Theorem \ref{Young_Love_estimate},
\begin{multline*}
 d(j_t,J_{s,\alpha,r}(F,w)) = 
 \inf_{f\in\mathcal S_{\alpha,r}(F)}
 \left\|\int_{0}^{t}f_t(u)\,dw(u) -\int_{0}^{s}f(u)\,dw(u)\right\|\\
 \begin{aligned}
 \leqslant &
 \left\|\int_{s}^{t}f_t(u)\,dw(u)\right\| +
 \inf_{f\in\mathcal S_{\alpha,r}(F)}
 \left\|\int_{0}^{s}(f - f_t)(u)\,dw(u)\right\|\\
 \leqslant &
 \mathfrak c_{\alpha,\beta}(T^{\alpha}\vee 1)\|w\|_{\beta,T}
 \left[N_{\alpha,T}(f_t)|t - s|^{\beta} +
 T^{\beta}\inf_{f\in\mathcal S_{\alpha,r}(F)}N_{\alpha,T}(f - f_t)
\right].
\end{aligned}
\end{multline*}
Since $f_t\in\mathcal S_{\alpha,r}(F)$, we have
$N_{\alpha,T}(f_t)\leqslant (\norm{F}_{\infty,T}+r)$
and the second term in the right-hand side of the previous inequality
is null.
Then,
\begin{displaymath}
d(j_t,J_{s,\alpha,r}(F,w))
\leqslant
\mathfrak c_{\alpha,\beta}(T^{\alpha}\vee 1)\|w\|_{\beta,T}
(\norm{F}_{\infty,T}+r)|t - s|^{\beta}
\end{displaymath}
and, by symmetry,
\begin{displaymath}
d(j_t,J_{s,\alpha,r}(F,w))
\leqslant
\mathfrak c_{\alpha,\beta}(T^{\alpha}\vee 1)\|w\|_{\beta,T}
(\norm{F}_{\infty,T}+r)|t - s|^{\beta}
\end{displaymath}
for every $j_t\in J_{t,\alpha,r}(F,w)$. Therefore,
\begin{multline*}
 d_\Haus(J_{s,\alpha,r}(F,w),J_{t,\alpha,r}(F,w))\\
 \begin{aligned}
 &=\max\left\{\sup_{j_s\in J_{s,\alpha,r}(F,w)}d(j_s,J_{t,\alpha,r}(F,w))\tq
 \sup_{j_t\in J_{t,\alpha,r}(F,w)}d(j_t,J_{s,\alpha,r}(F,w))\right\}\\
 &\leqslant
 \Cte_{\alpha,\beta,T}\|w\|_{\beta,T}(\norm{F}_{\infty,T}+r)|t - s|^{\beta}.
\end{aligned}
\end{multline*}
\end{proof}
%

% Corollary : Upper bounds on the Aumann-Young integral.

%
\begin{corollary}[Upper bound for
$\norm{J_{.,\alpha,r}(F,w)}_{\alpha,T}$ and $N_{\alpha,T}(J_{.,\alpha,r}(F,w))$]
\label{cor:alpha-norm_bound}
Assume that $r\geqslant\rmin$, and let
\begin{displaymath}
\rho_w(T,r,\|F\|_{\infty,T}) :=
\Cte_{\alpha,\beta,T}(\norm{F}_{\infty,T}+r)\|w\|_{\beta,T}T^{\beta -\alpha}.
\end{displaymath}
Then
\begin{equation*}
\|J_{.,\alpha,r}(F,w)\|_{\alpha,T}
\leqslant
\rho_w(T,r,\|F\|_{\infty,T})
\end{equation*}
and
\begin{equation*}
N_{\alpha,T}(J_{.,\alpha,r}(F,w))
\leqslant
(1+T^\alpha)\rho_w(T,r,\|F\|_{\infty,T}).
\end{equation*}
\end{corollary}
%

% Proof.

%
\begin{proof}
Since $\alpha <\beta$, the first inequality
is an immediate consequence of
Proposition \ref{Holder_continuity_Aumann_Young_integral}
and the obvious inequality
$\norm{F}_{\alpha,T}\leqslant T^{\beta -\alpha}\norm{F}_{\beta,T}$ for any
$F\in\C^{\beta\textrm{-H\"ol}}([0,T];$\allowbreak$\cc(\R^e))$.
The second inequality follows, using that $\int_0^0f(s)\,dw(s)=0$ for
any $f\in\mathcal S_{\alpha,r}(F)$.
\end{proof}
%
% \begin{notation}
%  In the sequel, $\rho_w(T,r)
%  :=\Cte_{\alpha,\beta,T}\|w\|_{\beta,T}T^{\beta -\alpha}r$,
% where $\Cte_{\alpha,\beta,T}$ is the constant defined in 
% Theorem \ref{Young_Love_estimate}.
% \end{notation} 
% Note that, 
% thanks to Proposition \ref{Holder_continuity_Aumann_Young_integral}, when $\alpha <\beta$,
% \begin{displaymath}
% N_{\alpha,T}(J_{.,\alpha,r}(F,w)) =
% \|J_{.,\alpha,r}(F,w)\|_{\alpha,T}
% \leqslant
% \rho_w(T,r).
% \end{displaymath}
%

% Section : Existence of fixed-points for functionals of Aumann-Young's integral and applications to differential inclusions.

%
\section{Existence of fixed-points for functionals of Aumann-Young's integral and applications to differential inclusions}\label{section_fixed_points}
The main theorem of this section, derived from
Kakutani-Fan-Glicksberg's theorem thanks to the results of Section
\ref{section_set_valued_Young_integral}, deals with the existence of
fixed-points for functionals of the Aumann-Young integral. We provide
some applications to differential inclusions, especially differential
inclusions driven by a fractional Brownian motion.

In the sequel,
$\alpha,\beta\in]0,1[$, $\alpha +\beta > 1$ and $\alpha <\beta$. This implies that $\beta >\frac{1}{2}.$ As usual, $T>0$ and $w\in \C^{\beta\textrm{-H\"ol}}([0,T];\mathbb R^d)$. 
%

% Subsection : Fixed point theorem.

%
\subsection{Fixed point theorem}
Let us first recall the notions of fixed-points and of upper semicontinuity
for multifunctions. 
%

% Definition : Fixed point.

%
\begin{definition}\label{fixed_point}
Consider a set $S$ and a multifunction $F : S\rightrightarrows S$. An element $x$ of $S$ is a \emph{fixed-point} of $F$ if $x\in F(x)$.
\end{definition}
We now give a definition of upper semicontinuity for multifunctions
in a particular case
which is sufficient for our needs 
(see \cite[Definition 1.4.1 and Proposition 1.4.8]{aubin-frankowska}).
%

% Definition : Upper-hemicontinuity.

%
\begin{definition}\label{upper_hemicontinuity}
Let $\mathbb X$ be a metric space, let $\mathbb Y$ be a compact metric
space.
Let $F :\mathbb X\rightrightarrows\mathbb Y$ be a multifunction with
closed values, and let $S$ be a closed subset of $\mathbb{X}$, with
$S\subset\Domain{F}:=\{x\in\mathbb{X}\tq F(x)\not=\emptyset\}$.
The multifunction $F$ is said to be
\emph{upper semicontinuous} on $S$ if,
for every sequence $(x_n,y_n)_{n\in\mathbb N}$ of elements of
$S\times\mathbb{Y}$ and for every  $(x,y)\in S\times\mathbb{Y}$
such that $(x_n,y_n)$ converges to $(x,y)$ and
$y_n\in F(x_n)$ for every $n$,
we have $y\in F(x)$.
\end{definition}
Let us now set the scene for the fixed point theorem. 
Let $S$ be a convex compact subset of
$\C^{\alpha\normalfont{\textrm{-H\"ol}}}([0,T];\R^e)$,
and let $\Phi : [0,T]\times\mathbb R^e\rightarrow
\cc(\matrl))$ be continuous for the Hausdorff distance,
with $\ell\in\mathbb N^*$.
Assume that there exists $R>0$ such that
\begin{equation*}
\Phi(.,x(.))\in
\overline B_{\alpha,\cc(M_{\ell,d}(\R))}(0,R)
\tq
\forall x\in S.
\end{equation*}
Let
\[r\geqslant 
\sup_{x\in S}\min\Bigl(\StLip{e\ell}\norm{\Phi(.,x(.))}_{\alpha,T},\,
\norm{\Phi(.,x(.))}_{\alpha,T,\Dem}\Bigr),\]
so that the map
\begin{displaymath}
\Phi_w :\,
\left\{\begin{array}{lcl}
S & \longrightarrow & \C^{\alpha\normalfont{\textrm{-H\"ol}}}([0,T];\cc(\R^\ell))\\
x & \longmapsto &\displaystyle{\Aint{\alpha,r}_0^{.}\Phi(s,x(s))\,dw(s)}
\end{array}\right.
\end{displaymath}
is well-defined.
With the notations of Corollary \ref{cor:alpha-norm_bound}, we have
\begin{equation*}
\norm{\Phi_w(x)}_{\alpha,T}
\leqslant
\rho_w(T,r,\|\Phi(.,x(.))\|_{\infty,T})
\leqslant
\rho_w(T,r,R)
\end{equation*}
for all $x\in S$. 
%

% Theorem : Existence of fixed points.

%
\begin{theorem}[Fixed point theorem]\label{existence_fixed_points}
Let $S$, $\Phi$, $r$ and $\Phi_w$ as above,
let $\alpha'\in ]0,\alpha[$, and let
\[\Psi :\C^{\alpha\normalfont{\textrm{-H\"ol}}}([0,T];\cc(\R^\ell))
\rightrightarrows
\C^{\alpha\normalfont{\textrm{-H\"ol}}}([0,T];\mathbb R^e),\]
a multifunction such that $\Domain{\Psi}$ contains
$\B:=\{\Phi_w(x)\tq x\in S\}$,
and which satisfies the following conditions:
\begin{enumerate}
 \item For every sequence $(F_n)_{n\in\mathbb N}$ of
 elements of $\B$ such that 
 there exists $F\in\B$ satisfying
 \begin{displaymath}
 \Ls_{n\rightarrow\infty}F_n(t)\subset F(t)
 \tq
 \forall t\in [0,T],
 \end{displaymath}
 if $\psi_n\in\Psi(F_n)$ converges in $\C^{\alpha'\normalfont{\textrm{-H\"ol}}}([0,T];\mathbb R^e)$ to $\psi\in \C^{\alpha\normalfont{\textrm{-H\"ol}}}([0,T];\mathbb R^e)$, then $\psi\in\Psi(F)$.
 \item For every $F\in\B$, $\Psi(F)$ is convex, closed and contained in $S$.
\end{enumerate}
Then, $\Gamma =\Psi\circ\Phi_w$ has at least one fixed-point in $S$.
\end{theorem}
%

% Remark : Domain of \Psi.

%
\begin{remark}\label{rem:DomPsi}
By Corollary \ref{cor:alpha-norm_bound},
in order that $\Domain(\Psi)\supset\B$, it is sufficient that
$\Domain(\Psi)\supset\overline B_{\alpha,\cc(M_{\ell,d}(\mathbb R))}(0,(1+T^\alpha)\rho_w(T,r,R))$.
\end{remark}
%

% Proof.

%
\begin{proof}[Proof of Theorem \ref{existence_fixed_points}]
By Condition 2, $\Gamma(x)$ is closed convex and contained in $S$ for every $x\in S$.

Let us check that $\Gamma$ is upper semicontinuous.
Let $(x_n)_{n\in\mathbb N}$ be a sequence of elements of $S$
converging to $x\in S$ in $\mathbb X :=\C^{\alpha'\textrm{-H\"ol}}([0,T];\mathbb R^e)$, and consider $\psi_n\in\Gamma(x_n) =\Psi(\Phi_w(x_n))$ converging to $\psi\in\C^{\alpha\textrm{-H\"ol}}([0,T];\mathbb R^e)$ in $\mathbb X$. By Proposition \ref{continuity_set_valued_Young_integral} (more precisely Remark \ref{rem-ppte-T}) and the Hausdorff continuity of $\Phi$ (which gives
\begin{displaymath}
\Ls_{n\rightarrow\infty}\Phi(t,x_n(t))\subset\Phi((t,x(t))
\tq
\forall t\in [0,T])),
\end{displaymath}
we have
\begin{displaymath}
\Ls_{n\rightarrow\infty}
\Phi_w(x_n)(t)\subset\Phi_w(x)(t)
\tq\forall t\in [0,T].
\end{displaymath}
Then, by Condition 1, $\psi\in\Psi(\Phi_w(x)) =\Gamma(x)$, which proves the upper semicontinuity.

We deduce by Kakutani-Fan-Glicksberg Theorem \cite[Theorem 8.6 of
II.§7]{GDfixed} that $\Gamma$ has at least one fixed-point in $S$.
\end{proof}
%

% Remark : Refined fixed point theorem.

% %%%%%%%%%%%%%%
\begin{remark}\label{rem:refined_fixed_point_th}
Consider $\gamma\in ]0,1\wedge(\beta/\alpha)]$ with % such that 
$\alpha\gamma+\beta > 1$.
The statement of Theorem \ref{existence_fixed_points} remains true when
\begin{equation*}
\Phi(.,x(.))\in
\overline B_{\alpha\gamma,\cc(M_{\ell,d}(\R))}(0,R)
\tq
\forall x\in S
\end{equation*}
and
\begin{displaymath}
\Phi_w(x) := \Aint{\alpha\gamma,r}_0^{.}\Phi(s,x(s))\,dw(s).
\end{displaymath}
For the sake of readability, this result has been detailed in the case $\gamma
= 1$, but the proof of Theorem \ref{existence_fixed_points} remains
unchanged for $\gamma\not=1$,
again with
$S\subset\C^{\alpha\normalfont{\textrm{-H\"ol}}}([0,T];\R^e)$, but 
replacing $\alpha$ by $\alpha\gamma$ everywhere else.

\end{remark}
%%%%%%%%%%%%

%
Let us provide two examples of multifunctions $\Psi$ fulfilling
Conditions 1 and 2 of Theorem \ref{existence_fixed_points}. These
examples will be the basis for our applications to differential
inclusions. 
%

% Example 1 : First example of multifunction \Psi.

%
\begin{example}\label{example_Psi_1}
Assume that $T$ satisfies $\rho_w(T,r,R)\leqslant\mathfrak r$ with
$\mathfrak r > 0$. With the notations of Theorem
\ref{existence_fixed_points}, let us show that $\Psi_h(.) := h
+\{x\in\mathcal S_{\alpha,\rho_w(T,r,R)}(\cdot) : x(0) = 0\}$, with
$\ell = e$ and $h\in
\C^{\alpha\normalfont{\textrm{-H\"ol}}}([0,T];\mathbb R^e)$, fulfills
Conditions 1 and 2 for $S = S_{h,\mathfrak r} :=\{x\in\overline
B_{\alpha,\mathbb R^e}(h,\mathfrak r) : x(0) = h(0)\}$ which is a
convex compact subset of
$\mathbb X := \C^{\alpha'\textrm{-H\"ol}}([0,T];\mathbb R^e)$ for
$0<\alpha'<\alpha$, thanks
to %\cite[Proposition 5.28]{FV10}:
Proposition \ref{prop:compactnessHolder}.

1. Let $(F_n)_{n\in\mathbb N}$ be a sequence in
$\overline B_{\alpha,\ccm}(0,(1+T^\alpha)\rho_w(T,r,R))$
such that there exists $F\in\overline B_{\alpha,\ccm}(0,(1+T^\alpha)\rho_w(T,r,R))$
satisfying
\begin{equation}\label{existence_fixed_points_1}
\Ls_{n\rightarrow\infty}F_n(t)\subset F(t)
\tq
\forall t\in [0,T].
\end{equation}
Consider also $\psi_n\in\Psi_h(F_n)$ converging in $\mathbb X$ to
$\psi\in \C^{\alpha\normalfont{\textrm{-H\"ol}}}([0,T];\mathbb
R^e)$. For any $n\in\mathbb N$, $(\psi_n - h)(0) = 0$ and, by the
definition of $\mathcal S_{\alpha,\rho_w(T,r,R)}(F_n)$, $\psi_n -
h\in\mathcal S^0(F_n)$ and $N_{\alpha,T}(\psi_n -
h)\leqslant\rho_w(T,r,R)$. Thanks to (\ref{existence_fixed_points_1}),
$\psi - h\in\mathcal S^0(F)$, and since $\psi$ is the limit of
$\psi_n$ in $\mathbb X$, $(\psi - h)(0) = 0$ and $N_{\alpha,T}(\psi -
h)\leqslant\rho_w(T,r,R)$. Therefore, $\psi\in\Psi_h(F)$.

2. For any $F\in\overline B_{\alpha,\ccm}(0,(1+T^\alpha)\rho_w(T,r,R))$,
since $\mathcal S_{\alpha,\rho_w(T,r,R)}(F)$ is convex (resp.~closed) by Proposition \ref{prop:basic} (resp.~Proposition \ref{basic_topological_properties_selections_set_Holder}), $\Psi_h(F)$ is convex (resp.~closed). Moreover, since $\rho_w(T,r,R)\leqslant\mathfrak r$,
\begin{align*}
 \Psi_h(F) & \subset
 h +\mathcal S^0(F)\cap\{x\in\overline B_{\alpha,\mathbb R^e}(0,(1+T^\alpha)\rho_w(T,r,R)) : x(0) = 0\}\\
 & \subset 
 h +\{x\in\overline B_{\alpha,\mathbb R^e}(0,\mathfrak r) : x(0) = 0\} = S_{h,\mathfrak r}.
\end{align*}
\end{example}
%

% Example 2 : Second example of multifunction \Psi.

%
\begin{example}\label{example_Psi_2}
Assume that $d =\ell = 2$, $e = 1$ and that $T$ satisfies
$\rho_w(T,r,R)\leqslant\mathfrak r$ with $\mathfrak r > 0$. Consider
$f\in \C^{\alpha\normalfont{\textrm{-H\"ol}}}([0,T];\mathbb R)$. With
the notations of Theorem \ref{existence_fixed_points} and Example
\ref{example_Psi_1}, let us show that $\Psi_{h,f}(.) := h +\{fx_1 -
x_2\tq x = (x_1,x_2)\in\Psi_0(.)\}$ fulfills Conditions 1 and 2 of Theorem \ref{existence_fixed_points}
for $S = S_{h,f,\mathfrak r} := h +\{fx_1 - x_2\tq x = (x_1,x_2)\in S_{0,\mathfrak r}\}$:

1. Let $(F_n)_{n\in\mathbb N}$ be a sequence in $\overline B_{\alpha,\cc(M_{1,2}(\mathbb R))}(0,(1+T^\alpha)\rho_w(T,r,R))$,
and let
$F\in\overline B_{\alpha,\cc(M_{1,2}(\mathbb R))}(0,(1+T^\alpha)\rho_w(T,r,R))$
satisfying
\begin{equation}\label{existence_fixed_points_2}
\Ls_{n\rightarrow\infty}F_n(t)\subset F(t)
\tq
\forall t\in [0,T].
\end{equation}
Consider also $\psi_n\in\Psi_{h,f}(F_n)$ converging in $\mathbb X$ to
$\psi\in \C^{\alpha\normalfont{\textrm{-H\"ol}}}([0,T];\mathbb
R)$. For any $n\in\mathbb N$, $\psi_n = h + fx_{1,n} - x_{2,n}$ with
$x_n = (x_{1,n},x_{2,n})\in\mathcal S^0(F_n)$ such that $x_n(0) = 0$
and $N_{\alpha,T}(x_n)\leqslant\rho_w(T,r,R)$. Then,
%by Friz and Victoir \cite[Proposition 5.28]{FV10},
by Proposition \ref{prop:compactnessHolder},
there exists a subsequence $(x_{n_k})_{k\in\mathbb N}$ of $(x_n)_{n\in\mathbb N}$ converging in $\mathbb X$ to $x = (x_1,x_2)\in \C^{\alpha\normalfont{\textrm{-H\"ol}}}([0,T];\mathbb R)$ such that $x(0) = 0$ and $N_{\alpha,T}(x)\leqslant\rho_w(T,r,R)$. Moreover, thanks to (\ref{existence_fixed_points_2}), $x\in\mathcal S^0(F)$. So, for every $t\in [0,T]$,
\begin{align*}
 \psi(t) & = 
 \lim_{n\rightarrow\infty}\psi_n(t) =
 \lim_{k\rightarrow\infty}\psi_{n_k}(t)\\
 & = 
 h(t) +
 f(t)\lim_{k\rightarrow\infty}x_{1,n_k}(t)
 -\lim_{k\rightarrow\infty}x_{2,n_k}(t) \\
 & = h(t) + f(t)x_1(t) - x_2(t).
\end{align*}
Therefore, $\psi\in\Psi_{h,f}(F)$.

2. For any $F\in\overline B_{\alpha,\cc(M_{1,2}(\mathbb R))}(0,(1+T^\alpha)\rho_w(T,r,R))$, since $\Psi_0(F)$ is convex (resp.~$\Psi_0(F)\subset S_{0,\mathfrak r}$), $\Psi_{h,f}(F)$ is convex (resp.~$\Psi_{h,f}(F)\subset S_{h,f,\mathfrak r}$). Moreover, the same arguments than in the previous step yield that $\Psi_{h,f}(F)$ is closed.
\end{example}
%

% Subsection : Applications to differential inclusions.

%
\subsection{Applications to differential inclusions}
Let us provide two applications of Theorem
\ref{existence_fixed_points} to differential inclusions. First, let
$\Phi :\,[0,T]\times\R^e\rightarrow\cc(M_{e,d}(\mathbb R))$ be a
multifunction such that, for every $t\in[0,T]$ and every $x\in\R^e$,
$\Phi(.,x)$ is $\alpha$-H\"older continuous with respect to the Hausdorff distance and $\Phi(t,.)$ is Lipschitz continuous with respect to the Hausdorff distance too,
that is,
there exist $k_1,k_2 > 0$ such that for every $s,t\in[0,T]$ and $x,y\in\R^e$,
\begin{equation}\label{assumption_Phi_1}
d_\Haus(\Phi(s,x),\Phi(t,x))\leqslant k_1|t - s|^\alpha
\quad\textrm{ and }\quad
d_\Haus(\Phi(t,x),\Phi(t,y))\leqslant k_2\|x - y\|.
\end{equation}
Assume also that $\Phi$ is bounded with respect to the Hausdorff distance, that is, there exists $R > 0$ such that
\begin{equation}\label{assumption_Phi_2}
\sup_{(t,x)\in [0,T]\times\R^e}\sup_{y\in\Phi(t,x)}\|y\|\leqslant R.
\end{equation}
Consider $w\in \C^{\beta\textrm{-H\"ol}}([0,T];\mathbb R^d)$ and an inclusion of the form
\begin{equation}\label{first_order_inclusion}
x(t)\in\xi +\Aint{\alpha,r}_0^t\Phi(s,x(s))\,dw(s)
\tq
t\in [0,T],
\end{equation}
where $r$ is large enough and the unknown function $x$ is in $\C^{\alpha\textrm{-H\"ol}}([0,T];\mathbb R^e)$.
%

% Corollary : Existence of solutions to Young differential inclusions.

%
\begin{corollary}[First order differential inclusion]
\label{existence_first_order_inclusion}
Assume that $0 <\alpha <\beta$, $\alpha +\beta >1$ and $r\geqslant r_0 := R + k_1 + k_2$. Then, the set of solutions to \eqref{first_order_inclusion} is nonempty. 
\end{corollary}
%

% Proof.

%
\begin{proof}
First,
with the notations of Example \ref{example_Psi_1},
for any $x\in S_{\xi,1}$, the map $s\mapsto\Phi(s,x(s))$ is
$\alpha$-H\"older continuous.
Precisely, for every $s,t\in [0,T]$,
\begin{eqnarray*}
 & &
 d_\Haus(\Phi(t,x(t)),\Phi(s,x(s)))\\
 & &
 \hspace{1.5cm}\leqslant
 d_\Haus(\Phi(t,x(t)),\Phi(s,x(t))) + d_\Haus(\Phi(s,x(t)),\Phi(s,x(s)))\\
 & &
 \hspace{1.5cm}\leqslant
 k_1|t - s|^{\alpha} + k_2\|x(t) - x(s)\|\\
 & &
 \hspace{1.5cm}\leqslant
 (k_1 + k_2\|x -\xi\|_{\alpha,T})|t - s|^{\alpha}\\
 & &
 \hspace{1.5cm}\leqslant
 (k_1 + k_2)|t - s|^{\alpha}
\end{eqnarray*}
and then,
\begin{displaymath}
N_{\alpha,T}(\Phi(.,x(.)))
\leqslant R + k_1 + k_2 = r_0.
\end{displaymath}
Let $1 > T_0 > 0$ be such that
$(1+T_0^\alpha)\rho_w(T_0,r_0,R)\leqslant 1$ (see Remark \ref{rem:DomPsi}).
Applying Theorem \ref{existence_fixed_points} on $[0,T_0]$, with $r = r_0$, $S
= S_{\xi,1}$ and $\Psi =\Psi_{\xi}$ (see Example \ref{example_Psi_1}),
shows that $\Gamma =\Psi\circ\Phi_w$ has a fixed point on $[0,T_0]$.
Since the definition of
$T_0$ is independent of $\xi$,
gluing solutions on successive intervals
provides a fixed point for $\Gamma$ on $[0,T]$, 
which is thus a solution to (\ref{first_order_inclusion}) on $[0,T]$.
\end{proof}
%

% Remark : First order inclusions under refined conditions.

%
\begin{remark}\label{rem:refined_inclusion_order_1}
Consider $\gamma\in ]0,1\wedge(\beta/\alpha)]$ such that $\alpha\gamma +\beta > 1$. Thanks to Remark \ref{rem:refined_fixed_point_th}, the statement of Corollary \ref{existence_fixed_points} remains true when, for every $t\in [0,T]$, $\Phi(t,.)$ is $\gamma$-H\"older continuous but not necessarily Lipschitz continuous.
\end{remark}
Now, for $e = 1$, consider $w_0\in
\C^{\beta\textrm{-H\"ol}}([0,T];\mathbb R)$ and a second order
inclusion of the form
\begin{equation}\label{second_order_inclusion}
x(t)\in\xi +
({\rm A}_{\alpha,\rho_{w_0}(T,r,R)})\int_{0}^{t}
\left[({\rm A}_{\alpha,r})\int_{0}^{s}\Phi(u,x(u))\,dw_0(u)\right]dw_0(s).
\end{equation}
For any $x\in\C^{\alpha\textrm{-H\"ol}}([0,T];\mathbb R)$ such that the Aumann-Young integral in Inclusion (\ref{second_order_inclusion}) is well defined, thanks to the integration by parts formula for Young's integral,
\begin{multline*}
 ({\rm A}_{\alpha,\rho_{w_0}(T,r,R)})\int_{0}^{t}\left[({\rm
 A}_{\alpha,r})\int_{0}^{s}\Phi(u,x(u))\,dw_0(u)\right]dw_0(s)\\
 \begin{aligned}
 =&\left\{
 \int_{0}^{t}\int_{0}^{s}\varphi(u)\,dw_0(u)\,dw_0(s)\tq \varphi\in\mathcal S_{\alpha,r}(\Phi(.,x(.)))
 \right\}\\
 =&\left\{
 w_0(t)\int_{0}^{t}\varphi(s)\,dw_0(s) -
 \int_{0}^{t}w_0(s)\varphi(s)\,dw_0(s)\tq \varphi\in\mathcal S_{\alpha,r}(\Phi(.,x(.)))
 \right\}\\
 =&\Biggl\{
 w_0(t)\left(\int_{0}^{t}\varphi(s)(dw_0(s),w_0(s)\,dw_0(s))\right)_{\!1}\\
 & -
 \left(\int_{0}^{t}\varphi(s)(dw_0(s),w_0(s)\,dw_0(s))\right)_{\!2}\tq \varphi\in\mathcal S_{\alpha,r}(\Phi(.,x(.)))
 \Biggr\},
 \end{aligned}
\end{multline*}
where, for $i=1,2$, 
$\left(\int_{0}^{t}\varphi(s)(dw_0(s),w_0(s)\,dw_0(s))\right)_{\!i}$ 
denotes the $i^{\text{th}}$ coordinate of
  $\int_{0}^{t}\varphi(s)(dw_0(s),w_0(s)\,dw_0(s))$.
Then, proving the existence of solutions to (\ref{second_order_inclusion}) amounts to prove that
\begin{displaymath}
\Gamma : x\longmapsto
\Psi_{h,w_0}\left(({\rm A}_{\alpha,r})\int_{0}^{.}\Phi(s,x(s))\,dw(s)\right)
\textrm{ with }
w :=\left(
w_0,\int_{0}^{.}w_0(s)\,dw_0(s)
\right)
\end{displaymath}
has fixed points.
%

% Corollary : Existence of solutions to 2nd order Young differential inclusions.

%
\begin{corollary}[Second order differential inclusion]
\label{existence_second_order_inclusion}
Assume that $\alpha <\beta$, $\alpha +\beta > 1$ and $r\geqslant r_{w_0,T} := k_1 + k_2(N_{\alpha,T}(w_0) + 1)$. Then, the set of solutions to \eqref{second_order_inclusion} is nonempty.
\end{corollary}
%

% Proof.

%
\begin{proof}
First, consider $x_{w_0,\xi} :=\xi + w_0x_1 - x_2$ with $x = (x_1,x_2)\in S_{0,1}$. The map $s\mapsto\Phi(s,x_{w_0,\xi}(s))$ is $\alpha$-H\"older continuous. Precisely, for every $s,t\in [0,T]$,
\begin{eqnarray*}
 & &
 d_\D(\Phi(t,x_{w_0,\xi}(t)),\Phi(s,x_{w_0,\xi}(s)))\\
 & &
 \hspace{2cm}\leqslant
 d_\D(\Phi(t,x_{w_0,\xi}(t)),\Phi(s,x_{w_0,\xi}(t))) \\
 & &
 \hspace{3cm}
 + d_\D(\Phi(s,x_{w_0,\xi}(t)),\Phi(s,x_{w_0,\xi}(s)))\\
 & &
 \hspace{2cm}\leqslant
 k_1|t - s|^{\alpha} + k_2|x_{w_0,\xi}(t) - x_{w_0,\xi}(s)|\\
 & &
 \hspace{2cm}\leqslant
 k_1|t - s|^{\alpha} + k_2(|w_0(t)(x_1(t) - x_1(s))| \\
 & &
 \hspace{3cm}
 + |(w_0(t) - w_0(s))x_1(s)| + |x_2(t) - x_2(s)|)\\
 & &
 \hspace{2cm}\leqslant
 (k_1 + k_2(N_{\alpha,T}(w_0) + 1))|t - s|^{\alpha}
\end{eqnarray*}
and then,
\begin{displaymath}
N_{\alpha,T}(\Phi(.,x(.)))
\leqslant k_1 + k_2(N_{\alpha,T}(w_0) + 1) = r_{w_0,T}.
\end{displaymath}
Let $T_0 > 0$ be such that $(1+T_0^\alpha)\rho_w(T_0,r_{w_0,T},R)\leqslant 1$.
By Theorem \ref{existence_fixed_points} applied on $[0,T_0]$,
with $r = r_{w_0,T}$, $S = S_{\xi,w_0,1}$ and $\Psi =\Psi_{\xi,w_0}$ (see Example \ref{example_Psi_2}), $\Gamma =\Psi\circ\Phi_w$ has a fixed point, which is thus a solution to (\ref{second_order_inclusion}) on $[0,T_0]$. Since the definition of $T_0$ is independent of $\xi$, $\Gamma$ has a fixed point, which is thus a solution to (\ref{second_order_inclusion}) on $[0,T]$.
\end{proof} 
Let us conclude with applications to stochastic inclusions. Consider a $(d - 1)$-dimensional fractional Brownian motion $B = (B(t))_{t\in [0,T]}$ of Hurst index $H\in (1/2,1)$, which is a centered Gaussian process such that
\begin{displaymath}
\mathbb E(B_i(s)B_j(t)) =
\frac{1}{2}(t^{2H} + s^{2H} - |t - s|^{2H})\delta_{i,j}
\end{displaymath}
for every $s,t\in [0,T]$ and $i,j\in\{1,\dots,d - 1\}$, and let $(\Omega,\mathcal A,\mathbb P)$ be the associated canonical probability space. By the Garcia-Rodemich-Rumsey lemma (see Nualart \cite[Lemma A.3.1]{NUALART06}), the paths of $B$ are $\beta$-H\"older continuous for any $\beta\in (0,H)$. Consider $r > 0$ and $\alpha\in (0,1)$ such that $\alpha +\beta > 1$ and $\alpha <\beta$. Then, for any measurable map $F :\Omega\rightarrow\overline B_{\alpha,\ccm}(0,r)$, one can define a set-valued stochastic integral of $F$ with respect to $B$ by
\begin{displaymath}
\left[\Aint{\alpha,r}_{0}^{t}F(s)dB(s)\right](\omega) :=
\Aint{\alpha,r}_{0}^{t}F(s,\omega)dB(s,\omega)
\tq 
\omega\in\Omega
\textrm{, }
t\in [0,T].
\end{displaymath}
This allows to consider the stochastic inclusion
\begin{equation}\label{SDI}
X(t)\in\xi +\Aint{\alpha,r}_{0}^{t}\Phi(s,X(s))\,dW(s)
\tq t\in [0,T],
\end{equation}
where $\Phi :\,[0,T]\times\R^e\rightarrow\cc(M_{e,d}(\mathbb R))$ fulfills Assumptions (\ref{assumption_Phi_1}) and (\ref{assumption_Phi_2}), and $W(t) := (t,B_1(t),\dots,B_{d - 1}(t))$ for every $t\in [0,T]$. By Corollary \ref{existence_first_order_inclusion}, for every $r\geqslant k_1 + k_2 + R$, Inclusion (\ref{SDI}) has at least one pathwise solution. One can also consider the one-dimensional second order stochastic inclusion
\begin{equation}\label{SDI_2nd_order}
X(t)\in\xi + ({\rm A}_{\alpha,\rho_B(T,r,R)})\int_{0}^{t}
\left[\Aint{\alpha,r}_{0}^{s}\Phi(u,X(u))dB(u)\right]dB(s)
\tq t\in [0,T].
\end{equation}
By Corollary \ref{existence_first_order_inclusion}, for every $r\geqslant k_1 + k_2(N_{\alpha,T}(B) + 1)$, Inclusion (\ref{SDI_2nd_order}) has at least one pathwise solution.
%

%%%%%%%%%%%%%%%% Acknowledgments.
\noindent
\textbf{Acknowledgments.}
We thank the reviewers for their careful reading and valuable comments
which helped improve 
%the readability of
this article.

This work was funded by RFBR and CNRS, project number PRC2767.
%This work was supported by the GdR TRAG. 
We also thank the GDR TRAG (CNRS) for its support.

%%%%%%%%%%%%%
% References.

%
%\bibliography{biblio}

\end{document}